\documentclass[11pt,twoside]{article}
\usepackage[a4paper, left=3cm, right=3cm, top=3cm]{geometry}
\usepackage[utf8]{inputenc} 
\usepackage[T1]{fontenc}
\usepackage{float} 
\usepackage[english]{babel}
\usepackage{lmodern} 
\usepackage{amssymb} 
\usepackage{amsthm}
\usepackage{amsmath}
\usepackage{doi}
\usepackage{mathrsfs}
\usepackage{hyperref}
\usepackage{authblk} 
\usepackage{listings}
\usepackage{accents}
\usepackage{subcaption}
\usepackage[numbered,framed]{matlab-prettifier}
\usepackage{wrapfig}
\usepackage{graphicx} 
\usepackage{paralist}
\usepackage{mathtools} 
\usepackage[hang,flushmargin]{footmisc}
\usepackage{array}
\usepackage{pgfplots}
\pgfplotsset{compat=1.3}
\usepackage{tikz}
\usepackage{bm}
\usepackage{nccmath}
\usepackage{theoremref}
\usepackage{geometry}
\usepackage{setspace}
\usepackage{fancyhdr}
\usepackage[makeroom]{cancel} 
\usepackage{calc}
\usepackage{url}
\usepackage[small]{titlesec} 
\usepackage{emptypage}	
\usepgfplotslibrary{fillbetween}
\usetikzlibrary{intersections}
\usepackage{fancyhdr}
\pgfdeclarelayer{bg}
\pgfsetlayers{bg,main}

\usepackage{xparse}
\NewDocumentCommand{\dcurl}{sO{}m}{%
  \IfBooleanTF{#1}
    {\dcurlext{#3}}
    {\dcurlx[#2]{#3}}%
}
\NewDocumentCommand{\dcurlext}{m}{%
  \sbox0{%
    \mathsurround=0pt
    $\left\{\vphantom{#1}\right.\kern-\nulldelimiterspace$%
  }%
  \sbox2{\{}%
  \ifdim\ht0=\ht2
    \{\kern-.625\wd2 \{#1\}\kern-.625\wd2 \}%
  \else
    \left\{\kern-.7\wd0\left\{#1\right\}\kern-.7\wd0\right\}%
  \fi
}
\NewDocumentCommand{\dcurlx}{om}{%
  \sbox0{\mathsurround=0pt$#1\{$}%
  \sbox2{\{}%
  \ifdim\ht0=\ht2
    \{\kern-.625\wd2 \{#2\}\kern-.625\wd2 \}%
  \else
    \mathopen{#1\{\kern-.7\wd0 #1\{}
    #2
    \mathclose{#1\}\kern-.7\wd0 #1\}}
  \fi
}
\NewDocumentCommand{\dsquare}{sO{}m}{%
  \IfBooleanTF{#1}
    {\dsquareext{#3}}
    {\dsquarex[#2]{#3}}%
}
\NewDocumentCommand{\dsquareext}{m}{%
  \sbox0{%
    \mathsurround=0pt
    $\left[\vphantom{#1}\right.\kern-\nulldelimiterspace$%
  }%
  \sbox2{[}%
  \ifdim\ht0=\ht2
    [\kern-.525\wd2 [#1]\kern-.525\wd2 ]%
  \else
    \left[\kern-.6\wd0\left[#1\right]\kern-.6\wd0\right]%
  \fi
}
\NewDocumentCommand{\dsquarex}{om}{%
  \sbox0{\mathsurround=0pt$#1[$}%
  \sbox2{[}%
  \ifdim\ht0=\ht2
    [\kern-.525\wd2 [#2]\kern-.525\wd2 ]%
  \else
    \mathopen{#1[\kern-.6\wd0 #1[}
    #2
    \mathclose{#1]\kern-.6\wd0 #1]}
  \fi
}

\newcommand{\dx}{\,\mathrm{d}x}
\newcommand{\ds}{\,\mathrm{d}s}
\newcommand{\dshat}{\,\mathrm{d}\hat{s}}

\newcommand{\dxtilde}{\,\mathrm{d}\tilde{x}}
\newcommand{\T}{^\mathsf{T}}
\newcommand{\Tinv}{^\mathsf{-T}}
\newcommand{\If}{\Gamma^{\mathrm{if}}}
\newcommand{\Ifh}{\Gamma^{\mathrm{if}}_h}
\newcommand{\Ifm}{\Gamma^{\mathrm{if}}_m}
\newcommand{\functions}[2]{\left\{{#1}\,\middle|\,{#2}\right\}}
\renewcommand{\div}{\text{div }}
\newcommand{\Shif}{\mathcal{S}_h^{\mathrm{if}}}
\newcommand{\Shb}{\mathcal{S}_h^{\mathrm{b}}}
\newcommand{\Shint}{\mathcal{S}_h^0}
\newcommand{\Sh}{\mathcal{S}_h}
\newcommand{\Ghb}{\Gamma_h^{\mathrm{b}}}
\newcommand{\Ghint}{\Gamma_h^0}
\newcommand{\Gh}{\Gamma_h}
\newcommand{\id}{\mathrm{id}}
\newcommand{\uproman}[1]{\uppercase\expandafter{\romannumeral#1}}

\makeatletter
\newcommand{\customlabel}[2]{%
\protected@write \@auxout {}{\string \newlabel {#1}{{#2}{}}}}
\makeatother

\definecolor{highlight}{RGB}{255,70,0}

\newtheorem{thm}{Theorem}
\newtheorem{lem}{Lemma}
\newtheorem{prop}{Proposition}
\newtheorem{rmk}{Remark}

\newcommand{\Addresses}{{
  \bigskip
  \small
  Sören Bartels, \textsc{Abteilung für angewandte Mathematik, Albert-Ludwigs-Universität Freiburg, Hermann-Herder-Str. 10, 79104 Freiburg im Breisgau, Germany}\par\nopagebreak
  \textit{Email address}: \texttt{bartels@mathematik.uni-freiburg.de}

  \medskip

  Andrea Bonito, \textsc{Texas A\&M University, College Station, TX 77843, USA}\par\nopagebreak
  \textit{Email address}: \texttt{bonito@tamu.edu}

  \medskip

  Philipp Tscherner, \textsc{Abteilung für angewandte Mathematik, Albert-Ludwigs-Universität Freiburg, Hermann-Herder-Str. 10, 79104 Freiburg im Breisgau, Germany}\par\nopagebreak
  \textit{Email address}: \texttt{philipp.tscherner@mathematik.uni-freiburg.de}
}}

\begin{document}
\thispagestyle{empty}
\begin{center}
	\textbf{\large{ERROR ESTIMATES FOR A LINEAR FOLDING MODEL}}\bigskip\\
	{\large Sören Bartels, Andrea Bonito and Philipp Tscherner}
\end{center}
\medskip

\begin{abstract}
An interior penalty discontinuous Galerkin method is devised to approximate minimizers of a linear folding model by discontinuous isoparametric finite element functions that account for an approximation of a folding arc. The numerical analysis of the discrete model includes an a priori error estimate in case of an accurate representation of the folding curve by the isoparametric mesh. Additional estimates show that geometric consistency errors may be controlled separately if the folding arc is approximated by piecewise polynomial curves. Various numerical experiments are carried out to validate the a priori error estimate for the folding model.
\end{abstract}
\medskip

\let\thefootnote\relax\footnotetext{
\textit{Date}: \today
}

\section{Model Problem}
Due to the appearance in natural processes and their importance to modern technical devices, foldable structures have attracted a lot of attention in recent decades. Applications include flapping mechanisms in biology \cite{folding_nature_1,folding_nature_2}, protein folding \cite{folding_protein_1,folding_protein_2}, movable structures in architecture \cite{folding_architecture_1,folding_architecture_2}, sheet (metal) pressing and wrapping \cite{sheet_pressing,sheet_wrapping} or origami and kirigami \cite{origami_1,origami_2,origami_3.2,origami_4}. We address in this article the numerical discretization of a linear folding model. The interior penalty discontinuous Galerkin method turns out to be a practical candidate for this purpose as gradient jumps of the deformation may be neglected along the interface, thereby allowing for the simulation of foldable configurations. 
A corresponding large deformation model has recently been derived via dimension reduction by \textsc{Bartels}, \textsc{Bonito} and \textsc{Hornung} \cite{bartels2021modeling}. The authors adapt arguments from the seminal work of \textsc{Friesecke}, \textsc{James} and \textsc{Müller} \cite{fjm} to account for the presence of a folding arc and follow ideas of \textsc{Bartels} \cite{bartels_dkt} and \textsc{Bonito}, \textsc{Nochetto} and \textsc{Ntogkas} \cite{Bonito_Bending} for the numerical realization.

To introduce the discontinuous Galerkin method, we follow the derivation of \cite{Bonito_Bending} for a classical bending problem and include a folding mechanism.
Let $\Omega\subset\mathbb{R}^2$ be a bounded polygonal Lipschitz domain and assume that $\Omega$ is partitioned into two subdomains $\Omega_1$ and $\Omega_2$ by an interface $\If$, in the sense that $\Omega=\Omega_1\cup\Omega_2\cup\If$, as shown in Figure \ref{domain}. We consider small displacements that are allowed to fold along the interface, giving rise to a linear bending problem. The problem is closely related to the linear Kirchhoff model, see e.g. \cite[Chapter 8]{nonlinear_bartels}, in which the deformation is assumed to be a perturbation of the identity in vertical direction.
For this type of small deflections the usual isometry constraint of the nonlinear model is negligible.
In particular, for a suitably rescaled body force $f\in L^2(\Omega)$, we seek a minimizer $u:\Omega\to\mathbb{R}$ of the continuous energy
\begin{align}\label{cont_energy}
E(u) = \frac{1}{2}\int_{\Omega\setminus\If} |D^2u|^2\dx-\int_\Omega f u\dx\,,
\end{align}
in the set of admissible functions
\begin{align}
\mathbb{V}(g,\Phi) := \functions{v\in H^2(\Omega_1\cup\Omega_2)\cap H^1(\Omega)}{v=g,\,\nabla v=\Phi\text{ on }\partial_D\Omega}\,.
\end{align}
The function space encodes clamped boundary conditions $u=g$ and $\nabla u=\Phi$ on a subset $\partial_D\Omega\subset\partial\Omega$. 
Here we assume that the boundary data $g\in H^{1/2}(\partial_D \Omega)$ and $\Phi\in [H^{1/2}(\partial_D \Omega)]^2$ are traces of functions  $g\in H^{1}(\Omega)$ and $\Phi\in [H^{1}(\Omega)]^2$. In addition, the configuration is allowed to fold along $\If$ since $u$ is only required to be piecewise in $H^2$ and globally in $H^1$.

The Euler--Lagrange equation for a minimizer $u\in\mathbb{V}(g,\Phi)$ of \eqref{cont_energy} is
\begin{align}\label{Euler_Lagrange}
\int_{\Omega\setminus\If}D^2u:D^2v\dx=\int_\Omega f v\dx\quad\forall v\in\mathbb{V}(0,0)\,.
\end{align}
The strong form of \eqref{Euler_Lagrange} reads
\begin{align}\label{strong_form}
\div\div D^2u=\Delta^2u=f\quad\text{in }\Omega\,, 
\end{align}
with natural boundary conditions
\begin{align}\label{natural_boundary}\tag{NBC}
\partial_\eta\nabla u = D^2u\,\eta=0,\quad
\partial_\eta\Delta u = (\div D^2u)\,\eta = 0
\quad\text{on }\partial\Omega\setminus\partial_D\Omega\,,
\end{align}
and the outward unit normal vector $\eta$ to $\Omega$, as well as natural interface conditions 
\begin{align}\label{natural_interface}\tag{NIC}
\dsquare{u}=0,\quad
{\partial_\eta\nabla u = 0},\quad
\dsquare{\partial_\eta\Delta u} = 0
\quad\text{on }\If\,,
\end{align}
with $\dsquare{u}:=u|_{\Omega_2}-u|_{\Omega_1}$ and the outward unit normal vector $\eta$ pointing from $\Omega_1$ into $\Omega_2$. 
The first interface term arises from the weak differentiability condition $u\in H^1(\Omega)$. Besides, for smooth test functions one expects $\dsquare{\partial_\eta\nabla u}=\dsquare{\partial_\eta\Delta u} = 0$ on $\If$. However, since we allow test functions $v\in\mathbb{V}(0,0)$ to have kinks across the interface $\If$, the second interface condition in \eqref{natural_interface} becomes $\partial_\eta\nabla u = 0$.
It infers that the curvature of the deformation along the fold vanishes in direction normal to the interface.

In what follows we assume that the geometry of the fold is compatible with the boundary conditions, such that there exists a unique solution to the problem. For instance, if the interface is straight and clamped boundary conditions are only imposed on one side of the fold, a solution on the other side of the fold might not be unique.

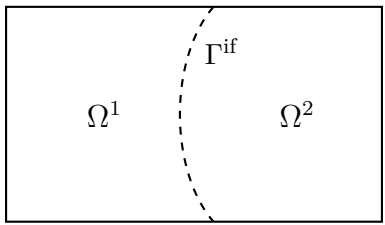
\begin{figure}[H]
\centering
	\begin{tikzpicture}[scale=1]
		\begin{axis}[axis lines=none, xmin=-3, xmax=3, width=9cm, height=5cm]
			\draw [samples=100, black, thick, domain=123.5:236.5, name path=C, dashed] plot (axis cs:{0.65+0.8*cos(\x)}, {0+0.6*sin(\x)});
			\addplot[black, thick, name path=A] coordinates { (0.21, -0.5) (2, -0.5) (2, 0.5) (0.21, 0.5)};
			\addplot[black, thick, name path=B] coordinates { (0.21, -0.5) (-2, -0.5) (-2, 0.5) (0.21, 0.5)};
			\node [right, black] at (axis cs:0.8, 0) {\large$\Omega^2$};
			\node [right, black] at (axis cs:-1.25, 0) {\large$\Omega^1$};
			\node [below, black] at (axis cs:0.30, 0.40) {$\If$};
		\end{axis}
	\end{tikzpicture}
	\caption{Sketch of the domain $\Omega=\Omega_1\cup\Omega_2\cup\If$ with curved interface $\If$, based on \cite{manolis}.}
	\label{domain}
\end{figure}

The outline of the article is as follows. In Section \ref{sec:discretization} we introduce the isoparametric interior penalty discontinuous Galerkin method and show existence and uniqueness of discrete solutions. The a priori error estimate is derived in Section \ref{sec:error_estimate}. Section \ref{sec:if_approx} addresses the polynomial approximation of the curved interface. Numerical experiments are stated in Section \ref{sec:experiments}.

\section{Discretization}\label{sec:discretization}
In this section we introduce the finite element discretization of the linear folding model. The isoparametric method accounts for an accurate representation of piecewise polynomial interfaces while the discontinuous Galerkin method allows us to simulate foldable structures when gradient jumps of the deformation are dropped along the folding curve. 

\subsection{Isoparametric Method}
We assume that every element of the triangulation $T\in\mathcal{T}_h$ is given by the image of a reference triangle $\widehat{T}$ under an isoparametric elemental diffeomorphism $\psi_T:\widehat{T}\to T$ of polynomial degree $k\in\mathbb{N}$, in particular $\psi_T\in[\mathbb{P}_k(\widehat{T})]^2$. Assuming shape regularity of the mesh provides that the derivative scales linearly with the diameter $h_T$ of the element $T$, in the sense that $||D\psi_T||_{L^\infty(\widehat{T})}\backsim h_T$ and $||D\psi_T^{-1}||_{L^\infty(T)}\backsim h_T^{-1}$. 
Proposition 1 from \textsc{Lenoir} \cite{Lenoir} (see also \textsc{Scott} \cite{Scott}) guarantees that we may choose mappings $\psi_T$, such that
\[
||D^m\psi_T||_{L^\infty(\widehat{T})}\lesssim h_T^{m}\quad\text{for }2\leq m\leq k+1\,.
\]
This is the case if nodes (i.e. the degrees of freedom) of the curved element $T$ are $h^m$-close to the corresponding nodes of the linear triangle $\widetilde{T}$ with the same vertices as $T$, see \textsc{Ciarlet} and \textsc{Raviart} \cite{Ciarlet_Raviart}.
Since $||D^2\psi_T||_{L^\infty(\widehat{T})}\lesssim||D\psi_T||_{L^\infty(\widehat{T})}$, we also conclude from \textsc{Ern} and \textsc{Guermond} \cite[Lemma 13.5]{ern_guermond} the converse estimate
\[
||D^m\psi_T^{-1}||_{L^\infty(T)}\lesssim h_T^{-m}\quad\text{for }2\leq m\leq k+1\,.
\]

For further aspects of isoparametric mappings we refer the reader to \cite{Ciarlet}. 
In what follows, we assume that the interface is accurately resolved by the isoparametric mesh, or equivalently, that $\If$ is given by the image of a piecewise polynomial map of degree $k\in\mathbb{N}$ from straight segments. This allows us to avoid geometrical errors in the a priori analysis. Corresponding errors are controlled separately, which is the subject of Section \ref{sec:if_approx}. 

The main feature of isoparametric methods is that the polynomial degree of the elemental mappings and that of the discrete functions coincide. With that in mind, we define the discrete space of discontinuous piecewise polynomials on a reference element $\widehat{T}$ by
\begin{align}
\mathbb{V}_h^k:=\functions{v\in L^2(\Omega)}{v|_T=\widehat{v}\circ\psi_T^{-1},\,\widehat{v}\in\mathbb{P}_k(\widehat{T})\;\;\forall T=\psi_T(\widehat{T})\in\mathcal{T}_h}\,.
\end{align}
Note that functions in $\mathbb{V}_h^k$ are in general not polynomials on the physical element $T$ if $k\geq2$.
We denote by $\Shint$ and $\Shb$ the set of edges contained in $\Omega$ (up to endpoints) and $\partial_D\Omega$, respectively. Contributions on $\partial\Omega\setminus\partial_D\Omega$ vanish due to the natural boundary conditions \eqref{natural_boundary}. Hence, the set of edges that enter the discontinuous Galerkin formulation is defined by $\Sh=\Shint\cup\Shb$. Since the (approximate) interface is accurately resolved by the mesh, we denote by $\Shif$ the set of edges contained in $\If$ (up to endpoints). The corresponding skeleton is denoted by $\Gh :=\cup\{S\in\Sh\}$. The sets $\Ghint,\Ghb$ and $\Ifh$ are defined similarly.

\subsection{Interior Penalty Discontinuous Galerkin Method}
To introduce the discontinuous Galerkin method we follow standard procedure from \cite{Bonito_Bending,dg_methods}. We denote by $\dsquare{\,\cdot\,}$ and $\dcurl{\cdot}$ the \emph{jump} and \emph{average} of a function $v$ over an inner side $S\in\Shint$ with $S=\partial T_+\cap\partial T_-$ and the unit normal $\eta$ pointing from $T_+$ into $T_-$. In particular, let
\begin{align}
\dsquare{v}|_S(x):=v|_{T_+}(x)-v|_{T_-}(x),\quad \dcurl{v}|_S(x):=\frac{1}{2}(v|_{T_+}(x)+v|_{T_-}(x))\,.
\end{align}
For boundary sides $S\in\Shb$ we set $\dcurl{v}|_S:=v|_S$ and consider the space
\begin{align}\label{discrete_space_boundary}
\mathbb{V}_h^k(g,\Phi):=\functions{v\in\mathbb{V}_h^k}{\dsquare{v}|_S=g-v,\,\dsquare{\nabla v}|_S=\Phi-\nabla v \text{ on every }S\in\Shb} 
\end{align}
to weakly enforce the Dirichlet boundary data. The sets $\mathbb{V}_h^k$ and $\mathbb{V}_h^k(g,\Phi)$ coincide but the latter contains a different interpretation of boundary jumps.

The elementwise applied differential operators are denoted by a subindex $h$. For example, the elementwise gradient is defined via $(\nabla_h v)|_T:=\nabla (v|_T)$.

To derive the bilinear form of the discontinuous Galerkin method, we multiply the strong form \eqref{strong_form} by a test function $v_h\in\mathbb{V}_h^k(0,0)$ and use elementwise integration by parts twice. Summing over all elements, we get
\begin{align}
\langle f,v_h\rangle_{L^2(\Omega)} = \sum_{T\in\mathcal{T}_h}\langle D^2u,D_h^2v_h\rangle_{L^2(T)}-\langle \partial_{\eta}\nabla u,\nabla_h v_h\rangle_{L^2(\partial T)}+\langle \partial_{\eta}\Delta u,v_h\rangle_{L^2(\partial T)}\,.
\end{align}
We collect the contributions over each side $S$ with a change in sign due to the definition of the jump and the normal $\eta$. Side terms on $\partial\Omega\setminus\partial_D\Omega$ vanish because of the natural boundary conditions \eqref{natural_boundary}. In addition, terms involving $\partial_\eta\nabla u$ vanish on $\Ifh$ due to the second interface condition \eqref{natural_interface}. Assuming that $u\in H^4(\Omega_1\cup\Omega_2)\cap H^1(\Omega)$ and incorporating again the interface conditions, the jumps $\dsquare{\partial_\eta\nabla u}$ and $\dsquare{\partial_\eta\Delta u}$ vanish on every edge $S\in\Sh\setminus\Shb$. On boundary sides $S\in\Shb$ we use the conventions $\dcurl{\partial_\eta\nabla u}=\partial_\eta\nabla u$ and $\dsquare{v}=-v$.
Applying the elementary formula $\dsquare{a\,b}=\dsquare{a}\dcurl{b}+\dsquare{b}\dcurl{a}$, we arrive at
\begin{align*}
\langle f,v_h\rangle_{L^2(\Omega)} = \langle D^2u,D_h^2v_h\rangle_{L^2(\Omega)}+\langle \dcurl{\partial_\eta\nabla u},\dsquare{\nabla_h v_h}\rangle_{L^2(\Gamma_h\setminus\Ifh)}-\langle \dcurl{\partial_\eta\Delta u},\dsquare{v_h}\rangle_{L^2(\Gamma_h)}\,.
\end{align*}
Using $\dsquare{u}=0$ on $\Gamma_h$ and $\dsquare{\nabla u}=0$ on $\Gamma_h\setminus\Ifh$, we add vanishing penalty terms to arrive at the following formulation including a symmetric bilinear form
\begin{align}\label{bilinearform}
\begin{split}
\langle f,v_h\rangle_{L^2(\Omega)} &= \langle D^2u,D_h^2v_h\rangle_{L^2(\Omega)}\\
&\quad+\langle \dcurl{\partial_\eta\nabla u},\dsquare{\nabla_h v_h}\rangle_{L^2(\Gamma_h\setminus\Ifh)}+\langle \dcurl{\partial_\eta\nabla_h v_h},\dsquare{\nabla u}\rangle_{L^2(\Gamma_h\setminus\Ifh)}\\
&\quad-\langle \dcurl{\partial_\eta\Delta u},\dsquare{v_h}\rangle_{L^2(\Gamma_h)}-\langle \dcurl{\partial_\eta\Delta_h v_h},\dsquare{u}\rangle_{L^2(\Gamma_h)}\\
&\quad+\gamma_1\langle h^{-1}\dsquare{\nabla u},\dsquare{\nabla_h v_h}\rangle_{L^2(\Gamma_h\setminus\Ifh)}+\gamma_0\langle h^{-3}\dsquare{u},\dsquare{v_h}\rangle_{L^2(\Gamma_h)}\\
&=:a_h(u,v_h)\,,
\end{split}
\end{align}
where the parameter $h$ is locally equivalent to the size $h_S$ of an edge $S$.
Hence, the discontinuous Galerkin method consists of finding $u_h\in\mathbb{V}_h^k(g,\Phi)$, such that
\begin{align}\label{dg_equation}
a_h(u_h,v_h)=\langle f,v_h\rangle_{L^2(\Omega)} \quad \forall v_h\in\mathbb{V}_h^k(0,0)\,.
\end{align}

The regularity of solutions to second order interface problems generally depends on geometric properties of the interface \cite{Petzoldt2001RegularityRF}. We implicitly assume that $\If$ divides $\Omega$ in such a way, that the solution of \eqref{strong_form} satisfies $u\in H^4(\Omega_1\cup\Omega_2)\cap H^1(\Omega)$. Weaker regularity assumptions would be sufficient to ensure the consistency of the method in the sense of the following proposition,  the additional regularity is required for the error analysis.

\begin{prop}[Galerkin orthogonality]\label{galerkin}
Assume that the solution of the strong form \eqref{strong_form} satisfies $u\in H^4(\Omega_1\cup\Omega_2)\cap H^1(\Omega)$, then we have
\begin{align}
	a_h(u-u_h,v_h)=0 \quad \forall v_h\in\mathbb{V}_h^k(0,0)\,.
\end{align}
\end{prop}
\begin{proof}
The assumed regularity and the interface conditions \eqref{natural_interface} provide $\partial_\eta\nabla u = 0$ on $\Ifh$, $\dsquare{u}=\dsquare{\partial_\eta\nabla u}=\dsquare{\partial_\eta\Delta u}=0$ on $\Gamma_h$ and $\dsquare{\nabla u}=0$ on $\Gamma_h\setminus\Ifh$, which yields $a_h(u,v_h)=\langle f,v_h\rangle_{L^2(\Omega)}$ for all $v_h\in\mathbb{V}_h^k(0,0)$.
\end{proof}
\begin{rmk}
For unfitted meshes a geometric consistency error enters the formulation. In the case of a piecewise polynomial approximation of the interface, the additional term can be controlled separately, see Section \ref{sec:if_approx}.
\end{rmk}

In what follows we consider homogeneous boundary data $(g,\Phi)=(0,0)$, which is justified below in the proof of \thref{coercivity}. The subindex $h$ of the elementwise applied differential operators is dropped for ease of notation.
Since the bilinear form $a_h$ is symmetric we define the dG-norm as
\begin{align*}\label{dg_norm}
	||v_h||_{\text{dG}}:=\Biggl(\sum_{T\in\mathcal{T}_h}||D^2 v_h||_{L^2(T)}^2+\sum_{S\in\mathcal{S}_h\setminus\Shif}\frac{\gamma_1}{h_S}||\dsquare{\nabla v_h}||^2_{L^2(S)}+\sum_{S\in\mathcal{S}_h}\frac{\gamma_0}{h_S^3}||\dsquare{v_h}||_{L^2(S)}^2\Biggr)^{1/2}\,,
\end{align*}
which is a norm on $\mathbb{V}_h^k$ for any $\gamma_0,\gamma_1>0$, provided that $\Ghb\cap(\partial\Omega_i/\partial\If)\neq\emptyset$ for all $i\in\{1,2\}$.
The following inverse estimates \cite[Lemma 9.1]{Bonito_Bending} ensure boundedness and coercivity of the bilinear form $a_h$ with respect to $||\,.\,||_{\text{dG}}$ on $\mathbb{V}_h^k$.

\begin{lem}[Inverse estimates]\thlabel{inverse_estimates}
For $T\in\mathcal{T}_h$ and an edge $S\in\mathcal{S}_h$ of $T$, we have for every $v_h\in\mathbb{V}_h^k$ the estimates
\begin{align}
||D^2v_h||_{L^2(S)}\leq c_1 h_S^{-1/2}||D^2v_h||_{L^2(T)}\,,\quad ||D^3v_h||_{L^2(S)}\leq c_0 h_S^{-3/2}||D^2v_h||_{L^2(T)}\,.
\end{align}
\begin{proof}
See \cite[Lemma 9.1]{Bonito_Bending}.
\end{proof}
\end{lem}

The inverse estimates allow for a bound of the jump-average terms by the remaining terms included in the bilinear form $a_h$. 

\begin{lem}[Jump-average bounds]\thlabel{jump_average_bounds}
Let $c_0$ and $c_1$ be the constants of the inverse estimates. 
(i) For $S\in\mathcal{S}_h\setminus\Shif$ and the edge patch $\omega(S)=T_+\cup T_-$ we have
\begin{align}
\int_S\dcurl{\partial_\eta\nabla v_h}\dsquare{\nabla w_h}\ds\leq\frac{c_1^2}{\gamma_1}||D^2v_h||^2_{L^2(\omega(S))}+\frac{\gamma_1}{4h_S}||\dsquare{\nabla w_h}||_{L^2(S)}^2\,.
\end{align}
(ii) For $S\in\mathcal{S}_h$ and the edge patch $\omega(S)=T_+\cup T_-$ we have
\begin{align}
\int_S\dcurl{\partial_\eta\Delta v_h}\dsquare{w_h}\ds\leq\frac{c_0^2}{\gamma_0}||D^2v_h||^2_{L^2(\omega(S))}+\frac{\gamma_0}{4h_S^3}||\dsquare{w_h}||^2_{L^2(S)}\,.
\end{align}
\begin{proof}
By Hölder's inequality and the inverse estimates of \thref{inverse_estimates}, we deduce that
\begin{align*}
\int_S\dcurl{\partial_\eta\nabla v_h}\dsquare{\nabla w_h}\ds &\leq ||\dcurl{D^2v_h}||_{L^2(S)}||\dsquare{\nabla w_h}||_{L^2(S)}\\
&\leq c_1h_S^{-1/2}||D^2v_h||_{L^2(\omega(S))}||\dsquare{\nabla w_h}||_{L^2(S)}\,.
\end{align*}
Similar arguments apply to the left-hand side of the second item. Young's inequality $ab\leq a^2+b^2/4$ yields the desired estimates.
\end{proof}
\end{lem}
 
With the jump-average bounds we conclude the coercivity and boundedness of the bilinear form $a_h$ with respect to $||\,.\,||_{\text{dG}}$ on $\mathbb{V}_h^k$. In particular, the discontinuous Galerkin method admits a unique solution $u_h\in\mathbb{V}_h^k(g,\Phi)$.

\begin{prop}[Existence and uniqueness]\thlabel{coercivity}
If $\gamma_0,\gamma_1$ are sufficiently large we have:

(i) The bilinear form $a_h$ is coercive, in the sense that
\begin{align*}
a_h(v_h,v_h)\geq \frac{1}{2}||v_h||_{\text{dG}}^2\quad\forall v_h\in\mathbb{V}_h^k\,.
\end{align*}
(ii) The bilinear form $a_h$ is continuous, i.e.,
\begin{align*}
a_h(v_h,w_h)\leq c||v_h||_{\text{dG}}||w_h||_{\text{dG}}\quad\forall v_h,w_h\in\mathbb{V}_h^k\,.
\end{align*}
(iii) The discontinuous Galerkin method admits a unique solution $u_h\in\mathbb{V}_h^k(g,\Phi)$, such that
\begin{align*}
a_h(u_h,v_h)=\ell_h(v_h) \quad\forall v_h\in\mathbb{V}_h^k\,.
\end{align*}
\begin{proof}
(i) Summing the jump-average bounds of \thref{jump_average_bounds} over all elements $T\in\mathcal{T}_h$ yields
\begin{align*}
2\sum_{S\in\mathcal{S}_h\setminus\Shif}\int_S\dcurl{\partial_\eta\nabla v_h}\dsquare{\nabla v_h}\ds\geq-\frac{6c_1^2}{\gamma_1}\sum_{T\in\mathcal{T}_h}||D^2v_h||^2_{L^2(T)}-\frac{1}{2}\sum_{S\in\mathcal{S}_h\setminus\Shif}\frac{\gamma_1}{h_S}||\dsquare{\nabla v_h}||_{L^2(S)}^2\,,
\end{align*}
where we used that each element has three sides. Analogously, we have
\begin{align*}
2\sum_{S\in\mathcal{S}_h}\int_S\dcurl{\partial_\eta\Delta v_h}\dsquare{v_h}\ds\geq-\frac{6c_0^2}{\gamma_0}\sum_{T\in\mathcal{T}_h}||D^2v_h||^2_{L^2(T)}-\frac{1}{2}\sum_{S\in\mathcal{S}_h}\frac{\gamma_0}{h_S^3}||\dsquare{v_h}||_{L^2(S)}^2\,.
\end{align*}
Finally, if $\gamma_0$ and $\gamma_1$ are sufficiently large such that $6c_0^2/\gamma_0\leq1/2$ and $6c_1^2/\gamma_1\leq1/2$, the jump-average terms can be absorbed and we conclude that
\begin{align*}
a_h(v_h,v_h)&\geq\frac{1}{2}\Biggl(\sum_{T\in\mathcal{T}_h}||D^2 v_h||_{L^2(T)}^2+\sum_{S\in\mathcal{S}_h\setminus\Shif}\frac{\gamma_1}{h_S}||\dsquare{\nabla v_h}||^2_{L^2(S)}+\sum_{S\in\mathcal{S}_h}\frac{\gamma_0}{h_S^3}||\dsquare{v_h}||_{L^2(S)}^2\Biggr)\\
&=\frac{1}{2}||v_h||_{\text{dG}}^2\,.
\end{align*}
(ii) The boundedness of $a_h$ follows by the inverse estimates of \thref{inverse_estimates}.\\
(iii) The Lax--Milgram lemma provides the existence of a unique solution for homogeneous boundary data $(g,\Phi)=(0,0)$. For non-homogeneous data $(g,\Phi)$, we require compatible boundary conditions, in the sense that $\Phi = \nabla g|_{\partial_D\Omega}$. As a result, we may assume that $u=\widetilde{u}+g$ and $\nabla u = \nabla \widetilde{u}+\Phi$, where $\widetilde{u}$ and $\nabla\widetilde{u}$ satisfy homogeneous boundary conditions. The Lax--Milgram lemma provides the existence of a unique solution $\widetilde{u}$ to a problem with a modified linear functional $\widetilde{\ell}_h$. In particular, problem \eqref{dg_equation} is well-posed.
\end{proof}
\end{prop}

To simplify the a priori error analysis, we include the linear terms involving the boundary data in the right hand side of equation $\eqref{dg_equation}$ to arrive at the equivalent formulation: Find $u_h\in\mathbb{V}_h^k(0,0)=\mathbb{V}_h^k$ such that
\begin{align}\label{dg_equation_mod}
a_h(u_h,v_h)=\ell_h(v_h) \quad\forall v_h\in\mathbb{V}_h^k\,,
\end{align}
with the linear form
\begin{align}\label{linearform}
\begin{split}
\ell_h(v_h) &:= \langle f,v_h\rangle_{L^2(\Omega)}-\langle \partial_\eta \nabla_h v_h,\Phi\rangle_{L^2(\Ghb)}+\langle \partial_\eta \Delta_h v_h,g\rangle_{L^2(\Ghb)}\\
&\,\quad+\gamma_1\langle h^{-1}\Phi,\nabla_h v_h\rangle_{L^2(\Ghb)}+\gamma_0\langle h^{-3}g,v_h\rangle_{L^2(\Ghb)}\,,
\end{split}
\end{align}
where we used the convention $\dsquare{v}|_S=-v|_S$ for $S\in\Shb$.

\begin{rmk}[Local discontinuous Galerkin method]
We note that by following ideas of \cite{bonito2020ldg}, we may introduce a reconstructed Hessian $H_h$ to obtain a discontinuous Galerkin method that is well posed for arbitrarily small penalty parameters $\gamma_0$ and $\gamma_1$. The operator $H_h$ is defined by combining the broken Hessian $D^2_h v_h$ with globalized liftings of both the jump of the broken gradient $[\nabla_hv_h]$ and the jump of the deformation $[v_h]$. This leads to additional technical difficulties in the error estimate. 
\end{rmk}

\section{A Priori Error Estimate}\label{sec:error_estimate}
To prove an interpolation estimate on curved triangles, we transform integrals back to a reference element, on which approximation results like the Bramble--Hilbert lemma can directly be applied. By similar transformations to reference sides (on which classical trace estimates apply), the result can then be extended to an interpolation estimate on curved sides. The following inequality from \textsc{Ern} and \textsc{Guermond} \cite[Lemma 13.5]{ern_guermond}
\begin{align}\label{iso_estimate_ern_guermond}
||D^m\psi_T^{-1}||_{L^\infty(T)}\lesssim ||D\psi_T^{-1}||^m_{L^\infty(T)} \quad\forall\,2\leq m\leq k+1\,,
\end{align}
ensures that appearing derivatives of the isoparametric mappings can be bounded by $||D^m\psi_T^{-1}||_{L^\infty(T)}\lesssim h_T^{-m}$ via the shape-regularity assumption. Repeated application of the chain and product rules then imply
\begin{align}\label{iso_estimate_elements_inv}
||D^m v||_{L^2(T)}\lesssim h_T^{1-m} \sum_{j=1}^m||D^j \widehat{v}||_{L^2(\widehat{T})}\,.
\end{align}
Owing to \cite[Proposition 1]{Lenoir}, the map $\psi_T$ can be chosen such that $||D^m\psi_T||_{L^\infty(\widehat{T})}\lesssim h_T^{m}$ for $2\leq m\leq k+1$, which leads to the estimate
\begin{align}\label{iso_estimate_elements}
||D^m \widehat{v}||_{L^2(\widehat{T})}\lesssim h_T^{m-1}\sum_{j=1}^m||D^j v||_{L^2(T)}\lesssim h_T^{m-1} ||v||_{H^m(T)}\,.
\end{align}

\begin{rmk}
In contrast to usual estimates for affine elemental mappings $\psi_T\in[\mathbb{P}_1(\widehat{T})]^2$ with $H^m$-seminorm on the right hand side, we have the full $H^m$-norm due to non-vanishing higher derivatives $D^m\psi_T \neq 0$ for $m\geq2$. This also affects the a priori error estimate derived below in \thref{error_estimate_thm}.
\end{rmk}

Let $\widehat{\mathcal{I}}:C^0(\widehat{T})\to\mathbb{V}_h^k(\widehat{T})$ denote the Lagrange interpolation operator of degree $k\geq1$ over $\widehat{T}$ and define the induced Lagrange interpolation operator $\mathcal{I}_h:C^0(T)\to\mathbb{V}_h^k(T)$ via
\begin{align}
\mathcal{I}_hv:=\widehat{\mathcal{I}}(v\circ\psi_T)\circ\psi_T^{-1}=\widehat{\mathcal{I}}\widehat{v}\circ\psi_T^{-1}\,.
\end{align} 
We utilize the following interpolation estimate \cite[Lemma 9.4]{Bonito_Bending}, which is originally formulated for curved quadrilaterals but applies to curved triangles analogously.

\begin{lem}[Interpolation estimate on curved triangles]\thlabel{interpolation_curved_elements}
Let $T\in\mathcal{T}_h$ be given by the image of an isoparametric mapping $\psi_T\in[\mathbb{P}_k(\widehat{T})]^2$ that satisfies
\begin{align*}
||D^m\psi_T||_{L^\infty(\widehat{T})}\lesssim h_T^{m}\quad \forall\,2\leq m\leq k+1\,.
\end{align*}
For $v\in H^{k+1}(T)$ and the (induced) Lagrange interpolant $\mathcal{I}_hv\in\mathbb{V}_h^k(T)$ we have
\begin{align}
h_T^{m-2}|v-\mathcal{I}_hv|_{H^m(T)}\lesssim h_T^{k-1}||v||_{H^{k+1}(T)}\quad\forall\,0\leq m\leq k+1\,.
\end{align}
\end{lem}

With the help of \thref{interpolation_curved_elements} we prove an interpolation estimate on curved sides. 

\begin{lem}[Interpolation estimate on curved sides]\thlabel{interpolation_curved_sides}
Let the assumptions of \thref{interpolation_curved_elements} be satisfied. For a side $S\in\mathcal{S}_h$ and an element $T\in\{T_+,T_-\}$ adjacent to $S$ we have
\begin{align}
h_{S}^{m-3/2}|v-\mathcal{I}_hv|_{H^m(S)}\lesssim h_{T}^{k-1}||v||_{H^{k+1}(T)}\quad{\forall\,0\leq m\leq k}\,,
\end{align}
where the $H^m(S)$-norm is defined via transformation to a flat domain.

\begin{proof}
For transformations to reference sides we make use of Nanson's formula \cite{nanson}, which states that the ratio between the measures on $\widehat{S}$ and $S$ is given by $\text{det}(D\psi_T)|D\psi_T^{-T}\eta_{\widehat{S}}|$. Using the trace inequality on reference sides with diameter $h_{\widehat{S}}$ gives
\begin{align}\label{trace_ref}
	||\widehat{v}||_{L^2(\widehat{S})}^2\lesssim||\widehat{v}||_{L^2(\widehat{T})}^2+||\nabla\widehat{v}||_{L^2(\widehat{T})}^2\,.
\end{align}
We now extend the classical trace estimate to the isoparametric case. Combining the above results, a transformation shows that
\begin{align}\label{iso_estimate_sides}
\begin{split}
||v||_{L^2(S)}^2&=\int_{\widehat{S}}\widehat{v}^2\det(D\psi_T)|D\psi_T^{-T}\eta_{\widehat{S}}|\dshat
\lesssim h_T\,||\widehat{v}||_{L^2(\widehat{S})}^2\\
&\stackrel{\mathclap{\eqref{trace_ref}}}{\lesssim}\, h_T\,||\widehat{v}||_{L^2(\widehat{T})}^2+h_T\,||\nabla\widehat{v}||_{L^2(\widehat{T})}^2
\,\stackrel{\mathclap{\eqref{iso_estimate_elements}}}{\lesssim}\, h_T^{-1}\,||v||_{L^2(T)}^2+h_T\,||\nabla v||_{L^2(T)}^2\,,
\end{split}
\end{align}
where $||\widehat{v}||_{L^2(\widehat{T})}^2\lesssim h_T^{-2}\,||v||_{L^2(T)}^2$.
Derivatives are estimated in a similar fashion via
\begin{align}\label{iso_estimate_sides_derivatives}
\begin{split}
||D^m v||_{L^2(S)}^2&\lesssim h_T^{1-2m} \sum_{j=1}^m||D^j \widehat{v}||_{L^2(\widehat{S})}^2\\
&\stackrel{\mathclap{\eqref{trace_ref}}}{\lesssim}\, h_T^{1-2m}\sum_{j=1}^{m+1}||D^j \widehat{v}||_{L^2(\widehat{T})}^2
\,\stackrel{\mathclap{\eqref{iso_estimate_elements}}}{\lesssim}\,\, \sum_{j=1}^{m+1}h_T^{2j-2m-1}\sum_{i=1}^j||D^i v||_{L^2(T)}^2\,.
\end{split}
\end{align}
With \thref{interpolation_curved_elements} we conclude
\begin{align*}
h_S^{-3}||v-\mathcal{I}_hv||_{L^2(S)}^2&\stackrel{\mathclap{\eqref{iso_estimate_sides}}}{\lesssim}\, h_T^{-4}\,||v-\mathcal{I}_hv||_{L^2(T)}^2+h_T^{-2}\,||\nabla(v-\mathcal{I}_hv)||_{L^2(T)}^2\\
&\lesssim h_T^{2k-2}||v||_{H^{k+1}(T)}^2\,,
\end{align*}
which gives the assertion for $m=0$. Analogously, we have for the derivatives
\begin{align*}
h_S^{2m-3}||D^m(v-\mathcal{I}_hv)||_{L^2(S)}^2&\stackrel{\mathclap{\eqref{iso_estimate_sides_derivatives}}}{\lesssim}\,\, \sum_{j=1}^{m+1}h_T^{2j-4}\sum_{i=1}^j||D^i (v-\mathcal{I}_hv)||_{L^2(T)}^2\\
&\lesssim \,\, \sum_{j=1}^{m+1}h_T^{2j-4}\sum_{i=1}^j h_T^{2k-2i+2}||v||_{H^{k+1}(T)}^2\,,
\end{align*}
for every $m\leq k$.
We use the fact, that $h_T\lesssim1$ and $i\leq j$ implies $h_T^{2(j-i)}\lesssim1$. Noting that the above estimate can thus be simplified, we arrive at
\begin{align*}
h_S^{2m-3}||D^m(v-\mathcal{I}_hv)||_{L^2(S)}^2\lesssim h_T^{2k-2}||v||_{H^{k+1}(T)}^2\,,
\end{align*}
which completes the proof.
\end{proof}
\end{lem}

\begin{thm}[Error estimate]\thlabel{error_estimate_thm}
Let $\mathcal{T}_h$ be a mesh given by the images of isoparametric mappings $\psi_T\in[\mathbb{P}_k(\widehat{T})]^2$ that resolve the interface $\If$ and satisfy
\begin{align*}
||D^m\psi_T||_{L^\infty(\widehat{T})}\lesssim h_T^{m}\quad 2\leq m\leq k+1\,,
\end{align*}
for every $T\in\mathcal{T}_h$. Define the maximum diameter of elements in $\mathcal{T}_h$ as $h:=\max_{T\in\mathcal{T}_h}h_T$. 

(i) For $k\geq3$ and $u\in H^{k+1}(\Omega_1\cup\Omega_2)\cap H^1(\Omega)$ we have
\begin{align}\label{error_estimate}
||u-u_h||_{dG}\lesssim h^{k-1}||u||_{H^{k+1}(\Omega_1\cup\,\Omega_2)}\,.
\end{align}

(ii) For $k=2$ and $u\in H^{4}(\Omega_1\cup\Omega_2)\cap H^1(\Omega)$ we have
\begin{align}\label{error_estimate_k2}
||u-u_h||_{dG}\lesssim h||u||_{H^{4}(\Omega_1\cup\,\Omega_2)}\,.
\end{align}
\begin{proof}
(i) We first consider the case $k\geq3$. Let $\mathcal{I}_h u\in\mathbb{V}_h^k$ be the Lagrange interpolant of $u$, which is well-defined since $H^2(T)\subset C^0(T)$. By the triangle inequality we have
\begin{align*}
||u-u_h||_{\text{dG}} \leq ||u-\mathcal{I}_h u||_{\text{dG}} + ||\mathcal{I}_h u-u_h||_{\text{dG}} 
\leq ||e_I||_{\text{dG}} + ||d_h||_{\text{dG}} \,,
\end{align*}
where $e_I = u-\mathcal{I}_h u$ and $d_h = u_h-\mathcal{I}_h u$. To estimate the dG-norm of $e_I$ we use \thref{interpolation_curved_elements,interpolation_curved_sides} to conclude
\begin{align*}
||e_I||_{\text{dG}}^2 &= \sum_{T\in\mathcal{T}_h}||D^2 e_I||_{L^2(T)}^2+\sum_{S\in\mathcal{S}_h\setminus\Shif}\frac{\gamma_1}{h_S}||\dsquare{\nabla e_I}||^2_{L^2(S)}+\sum_{S\in\mathcal{S}_h}\frac{\gamma_0}{h_S^3}||\dsquare{e_I}||_{L^2(S)}^2\\
&\lesssim h^{2k-2}||u||_{H^{k+1}(\Omega_1\cup\,\Omega_2)}^2\,.
\end{align*}
To estimate the dG-norm of $d_h$ we use the Galerkin orthogonality of Proposition \ref{galerkin} to infer
\begin{align*}
a_h(d_h,v_h)=a_h(u_h-\mathcal{I}_h u,v_h)=a_h(u-\mathcal{I}_h u,v_h)=a_h(e_I,v_h)\quad\forall v_h\in\mathbb{V}_h^k\,.
\end{align*}
As a consequence, coercivity of $a_h$ and the Cauchy-Schwarz inequality yield 
\begin{align*}
\frac{1}{2}||d_h||_{\text{dG}}^2&\leq a_h(d_h,d_h)=a_h(e_I,d_h)\\[5pt]
&\leq \sum_{T\in\mathcal{T}_h}||D^2 e_I||_{L^2(T)} ||D^2 d_h||_{L^2(T)}\\
&\quad+ \sum_{S\in\mathcal{S}_h\setminus\Shif} ||\dcurl{\partial_\eta\nabla e_I}||_{L^2(S)} ||\dsquare{\nabla d_h}||_{L^2(S)} + ||\dcurl{\partial_\eta\nabla d_h}||_{L^2(S)} ||\dsquare{\nabla e_I}||_{L^2(S)} \\
&\quad+ \sum_{S\in\mathcal{S}_h} ||\dcurl{\partial_\eta\Delta e_I}||_{L^2(S)} ||\dsquare{d_h}||_{L^2(S)} + ||\dcurl{\partial_\eta\Delta d_h}||_{L^2(S)} ||\dsquare{e_I}||_{L^2(S)} \\
&\quad+ \sum_{S\in\mathcal{S}_h\setminus\Shif}\frac{\gamma_1}{h_S}||\dsquare{\nabla e_I}||_{L^2(S)} ||\dsquare{\nabla d_h}||_{L^2(S)} + \sum_{S\in\mathcal{S}_h}\frac{\gamma_0}{h_S^3}||\dsquare{e_I}||_{L^2(S)} ||\dsquare{d_h}||_{L^2(S)}\,.
\end{align*}
The inverse estimates of \thref{inverse_estimates} imply
\begin{align*}
||\dcurl{\partial_\eta\nabla d_h}||_{L^2(S)}&\leq c h_S^{-1/2}||D^2 d_h||_{L^2(T)}\,,\\
||\dcurl{\partial_\eta\Delta d_h}||_{L^2(S)}&\leq c h_S^{-3/2}||D^2 d_h||_{L^2(T)}\,,
\end{align*}
where $T\in\{T_+,T_-\}$ is the element with the larger contribution.
Hence,
\begin{align*}
||\dcurl{\partial_\eta\nabla d_h}||_{L^2(S)} ||\dsquare{\nabla e_I}||_{L^2(S)} &\leq \frac{c^2}{h_S}||\dsquare{\nabla e_I}||_{L^2(S)}^2 + \frac{1}{4}||D^2 d_h||_{L^2(T)}^2\,,\\
||\dcurl{\partial_\eta\Delta d_h}||_{L^2(S)} ||\dsquare{e_I}||_{L^2(S)} &\leq \frac{c^2}{h_S^3}||\dsquare{e_I}||_{L^2(S)}^2 + \frac{1}{4}||D^2 d_h||_{L^2(T)}^2\,,
\end{align*}
by Young's inequality $ab\leq a^2+b^2/4$. Similarly, we obtain
\begin{align*}
||\dcurl{\partial_\eta\nabla e_I}||_{L^2(S)} ||\dsquare{\nabla d_h}||_{L^2(S)} &\leq \frac{h_S}{\gamma_1}||\dcurl{\partial_\eta\nabla e_I}||_{L^2(S)}^2 + \frac{\gamma_1}{4h_S}||\dsquare{\nabla d_h}||_{L^2(S)}^2\,,\\
||\dcurl{\partial_\eta\Delta e_I}||_{L^2(S)} ||\dsquare{d_h}||_{L^2(S)} &\leq \frac{h_S^3}{\gamma_0}||\dcurl{\partial_\eta\Delta e_I}||_{L^2(S)}^2 + \frac{\gamma_0}{4h_S^3}||\dsquare{d_h}||_{L^2(S)}^2\,,
\end{align*}
as well as
\begin{align*}
\frac{\gamma_1}{h_S}||\dsquare{\nabla e_I}||_{L^2(S)} ||\dsquare{\nabla d_h}||_{L^2(S)} &\leq \frac{\gamma_1}{h_S}||\dsquare{\nabla e_I}||_{L^2(S)}^2 + \frac{\gamma_1}{4h_S}||\dsquare{\nabla d_h}||_{L^2(S)}^2\,,\\
\frac{\gamma_0}{h_S^3}||\dsquare{e_I}||_{L^2(S)} ||\dsquare{d_h}||_{L^2(S)} &\leq \frac{\gamma_0}{h_S^3}||\dsquare{e_I}||_{L^2(S)}^2 + \frac{\gamma_0}{4h_S^3}||\dsquare{d_h}||_{L^2(S)}^2\,,
\end{align*}
and lastly
\begin{align*}
||D^2 e_I||_{L^2(T)} ||D^2 d_h||_{L^2(T)} \leq ||D^2 e_I||_{L^2(T)}^2 +\frac{1}{4}||D^2 d_h||_{L^2(T)}^2\,.
\end{align*}
The sum of all terms involving $d_h$ coincides with $(1/4)||d_h||_{\text{dG}}^2$ and can be absorbed on the left-hand side $(1/2)||d_h||_{\text{dG}}^2$. A final application of \thref{interpolation_curved_elements,interpolation_curved_sides} on the remaining terms involving $e_I$ yields
\begin{align*}
||d_h||_{\text{dG}}^2&\lesssim \sum_{T\in\mathcal{T}_h}||D^2 e_I||_{L^2(T)}^2
+ \sum_{S\in\mathcal{S}_h\setminus\Shif} h_S||\dcurl{\partial_\eta\nabla e_I}||_{L^2(S)}^2 + h_S^{-1}||\dsquare{\nabla e_I}||_{L^2(S)}^2 \\
&\quad+ \sum_{S\in\mathcal{S}_h} h_S^3||\dcurl{\partial_\eta\Delta e_I}||_{L^2(S)}^2 + h_S^{-3}||\dsquare{e_I}||_{L^2(S)}^2\\[5pt]
&\lesssim h^{2k-2}||u||_{H^{k+1}(\Omega_1\cup\,\Omega_2)}^2\,,
\end{align*}
which shows the claim for $k\geq3$.

(ii) For the case $k=2$, \thref{interpolation_curved_elements,interpolation_curved_sides} can be applied to terms up to second order $m\leq2$
\begin{align*}
||e_I||_{\text{dG}}^2+\sum_{S\in\mathcal{S}_h\setminus\Shif} h_S||\dcurl{\partial_\eta\nabla e_I}||_{L^2(S)}^2 &\lesssim h^{2}||u||_{H^{3}(\Omega_1\cup\,\Omega_2)}^2\lesssim h^{2}||u||_{H^{4}(\Omega_1\cup\,\Omega_2)}^2\,.
\end{align*}
However, \thref{interpolation_curved_sides} does not apply to the third order side term $||\dcurl{\partial_\eta\Delta e_I}||_{L^2(S)}$ due to the restriction $m\leq k$. Nevertheless, we may take a similar workaround as in Proposition 4.4 from \cite{Bonito_Bending}. By the same arguments as in the proof of \thref{interpolation_curved_sides} we have
\begin{align*}
h_S^3||\dcurl{\partial_\eta\Delta e_I}||^2_{L^2(S)}\lesssim h_T^{-2}\Bigl(\sum_{j=1}^3||D^j \widehat{e_I}||^2_{L^2(\widehat{T})}+||D^4 \widehat{e_I}||^2_{L^2(\widehat{T})}\Bigr)\,,
\end{align*}
where we split the highest order term from the lower order ones. By a transformation to the element $T$, the Bramble--Hilbert lemma and using that $||u||_{H^3(T)}\lesssim||u||_{H^4(T)}$ we have
\begin{align*}
h_T^{-2}\sum_{j=1}^3||D^j \widehat{e_I}||^2_{L^2(\widehat{T})}\stackrel{\mathclap{\eqref{iso_estimate_elements}}}{\lesssim}\sum_{j=1}^3h_T^{2j-4}\sum_{i=1}^j||D^i e_I||_{L^2(T)}^2\lesssim h_T^2||u||_{H^4(T)}^2\,.
\end{align*}
To treat the fourth order term, we add and subtract the Lagrange interpolant $\widehat{\mathcal{I}}^3\widehat{u}$
\begin{align*}
||D^4(\widehat{u}-\widehat{\mathcal{I}}\widehat{u})||^2_{L^2(\widehat{T})}\lesssim||D^4(\widehat{u}-\widehat{\mathcal{I}}^3\widehat{u})||^2_{L^2(\widehat{T})}+||D^4(\widehat{\mathcal{I}}^3\widehat{u}-\widehat{\mathcal{I}}\widehat{u})||^2_{L^2(\widehat{T})}\,.
\end{align*}
Again, by the Bramble-Hilbert lemma and \eqref{iso_estimate_elements} we have $||D^4(\widehat{u}-\widehat{\mathcal{I}}^3\widehat{u})||^2_{L^2(\widehat{T})}\lesssim h_T^6||u||_{H^4(T)}^2$. Since the second norm only involves polynomials on $\widehat{T}$, we use an inverse estimate 
\begin{align*}
||D^4(\widehat{\mathcal{I}}^3\widehat{u}-\widehat{\mathcal{I}}\widehat{u})||^2_{L^2(\widehat{T})}\lesssim ||D^3(\widehat{\mathcal{I}}^3\widehat{u}-\widehat{\mathcal{I}}\widehat{u})||^2_{L^2(\widehat{T})}\,.
\end{align*}
Next, we add and subtract $\widehat{u}$ and apply the Bramble-Hilbert lemma and \eqref{iso_estimate_elements} to deduce
\begin{align*}
||D^4(\widehat{\mathcal{I}}^3\widehat{u}-\widehat{\mathcal{I}}\widehat{u})||^2_{L^2(\widehat{T})}&\lesssim ||D^3(\widehat{u}-\widehat{\mathcal{I}}^3\widehat{u})||^2_{L^2(\widehat{T})}+||D^3(\widehat{u}-\widehat{\mathcal{I}}\widehat{u})||^2_{L^2(\widehat{T})}\\
&\lesssim h_T^6||u||_{H^4(T)}^2+h_T^4||u||_{H^3(T)}^2\,.
\end{align*}
Combining the previous estimates, we arrive at
\begin{align*}
h_S^3||\dcurl{\partial_\eta\Delta e_I}||^2_{L^2(S)}&\lesssim h_T^{-2}\Bigl(\sum_{j=1}^3||D^j \widehat{e_I}||^2_{L^2(\widehat{T})}+||D^4 \widehat{e_I}||^2_{L^2(\widehat{T})}\Bigr)\\
&\lesssim h_T^{-2}\Bigl(h_T^4||u||_{H^4(T)}^2+h_T^6||u||_{H^4(T)}^2+h_T^4||u||_{H^3(T)}^2\Bigr)\\[5pt]
&\lesssim h^2||u||^2_{H^4(T)}\,,
\end{align*}
which completes the proof.
\end{proof} 
\end{thm}

%%%%%%%%%%%%%%%%%%%%%%%%%%%%%%%%%%%%%%%%%%%%%%%%%%%%%%%%%%%%%%%%%%%%%%%%%%%%%%

\section{Interface Approximation}\label{sec:if_approx}
In practice, the interface is generally not given as a union of images of polynomial maps. This section addresses the piecewise polynomial approximation of the interface. There are many results on the approximation of domains involving curved boundaries, e.g. for the Laplace--Beltrami operator \cite{BONITO20201} or for fourth order Kirchhoff plate bending problems \cite{Arnold_Walker}. We follow ideas for isoparametric finite elements \cite{Lenoir} by \textsc{Lenoir}. The aim is to compare the exact solution defined on the curved domain $\Omega$ with the transformed solution defined on an approximation $\Omega^m$ of $\Omega$ that consists of polynomial elements $T^m$ of order $m\geq1$ with maximum diameter $h$. Our error estimates then rely on the existence of a map $\Psi_h:\Omega^m\to\Omega$ satisfying for every $T^m$ the estimate
\begin{align}\label{interface_approx1}
	||D^s(\Psi_h-\id_{T^m})||_{L^\infty(T^m)}\lesssim h^{m+1-s}\quad\forall\,0\leq s\leq m+1\,.
\end{align}
An explicit formula for the map is given in \cite[eq. (32)]{Lenoir}. In case of the Poisson problem $-\Delta u=f$ with $u=0$ on $\partial\Omega$, we denote $\widetilde{u}=u\circ\Psi_h$ and $J=\nabla\Psi_h$. Equation \eqref{interface_approx1} gives an $h^m$-bound on the right-hand side of the following identity
\[
\int_{\Omega}\nabla u\cdot\nabla v\dx-\int_{\Omega^m}\nabla\widetilde{u}\cdot\nabla\widetilde{v}\dxtilde=\int_{\Omega^m}\nabla\widetilde{u}\cdot[\text{det}(J)J^{-1}J^{-\top}-I]\nabla\widetilde{v}\dxtilde\,,
\]
which is also called the geometric consistency error.

\medskip
In case of an internal interface approximation, we assume that $\mathcal{T}_h$ and $\mathcal{T}^m_h$ are triangulations of $\Omega$, such that possibly non-polynomial elements $T\in\mathcal{T}_h$ resolve the interface $\If$ exactly and such that the interface is approximated by sides of elements $T^m\in\mathcal{T}^m_h$ of polynomial order $m\geq1$ with nodes (degrees of freedom) on $\If$, thus forming an approximate interface $\Ifm$. From \cite{Lenoir} we know, that there exist local mappings $\Psi_T:T^m\to T$ between the triangulations, satisfying the following key properties. The collection of these local maps defines a global map $\Psi_h:\mathcal{T}_h^m\to\mathcal{T}_h$.
\begin{lem}[Properties of $\Psi_T$]
(i) There exist local mappings $\Psi_T:T^m\to T$ satisfying
\begin{align}
	||D^s(\Psi_T-\id_{T^m})||_{L^\infty(T^m)}&\lesssim h_T^{m+1-s}\quad\forall\,0\leq s\leq m+1\,,\label{interface_approx}\\
	||D^s(\Psi_T^{-1}-\id_{T})||_{L^\infty(T)}&\lesssim h_T^{m+1-s}\quad\forall\,0\leq s\leq m+1\,.\label{interface_approx_inv}
\end{align}
(ii) The maps $\Psi_T:T^m\to T$ satisfy
\begin{align}
	\sup_{x\in T^m}|\det(D\Psi_T)-1\,|&\lesssim h_T^{m}\,,\label{interface_approx_det}\\
	\sup_{x\in T}|\det(D\Psi_T^{-1})-1\,|&\lesssim h_T^{m}\,.\label{interface_approx_inv_det}
\end{align}
(iii) For every $v\in H^k(T)$ and $\widetilde{v}=v\circ\Psi_T$ we have equivalence of the norms
\begin{align}\label{norm_equivalence}
||\widetilde{v}||_{H^k(T^m)}\sim||v||_{H^k(T)}\,.
\end{align}
\begin{proof}
(i) The first item follows by arguments of \cite{Lenoir}.

(ii) The second item is a consequence of (i).

(iii) The norm equivalence follows by a transformation formula, using that $||D\Psi_T^{-1}||_{L^\infty(T)}$ and $||\det(D\Psi_T)||_{L^\infty(T^m)}$ are uniformly bounded with respect to $h$ (see equation (SM2.10) in the supplementary materials of \cite{Arnold_Walker}).
\end{proof}
\end{lem}

We require clamped boundary conditions $u=g$ and $\nabla u=\Phi$ on $\partial\Omega$ and denote 
\begin{align}
\mathbb{A}(g,\Phi)&:=\functions{v\in H^2(\Omega\setminus\If)\cap H^1(\Omega)}{v=g,\nabla v=\Phi\hspace{0.5em}\text{on }\partial\Omega}\,,\\
\mathbb{A}_m(g,\Phi)&:=\functions{v\in H^2(\Omega\setminus\Ifm)\cap H^1(\Omega)}{v=g,\nabla v=\Phi\hspace{0.5em}\text{on }\partial\Omega}\,.
\end{align}
Consider the following problem: Find $u\in \mathbb{A}(g,\Phi)$, such that for every $v\in \mathbb{A}(0,0)$ we have
\begin{align}\label{interface_problem}\tag{P}
\int_{\Omega\setminus\If}D^2u:D^2v\dx=\int_\Omega f v\dx\,.
\end{align}
Furthermore, we define the approximate problem: Find $\widetilde{u}_m\in \mathbb{A}_m(g,\Phi)$, such that for every $\widetilde{v}\in \mathbb{A}_m(0,0)$ we have
\begin{align}\label{interface_problem_approx}\tag{P\textsubscript{m}}
\int_{\Omega\setminus\Ifm}D^2\widetilde{u}_m:D^2\widetilde{v}\dxtilde=\int_\Omega \widetilde{f}_m \widetilde{v}\dxtilde\,.
\end{align}

Error estimates rely on the following result for the geometric consistency error.

\begin{lem}[Geometric consistency error]
With the above assumptions, we have
\begin{align}\label{consistency_error}
\left|\int_T D^2 v:D^2w\dx-\int_{T^m}D^2\widetilde{v}:D^2\widetilde{w}\dxtilde\right|\lesssim h^{m-1}||v||_{H^2(T)}||w||_{H^2(T)}
\end{align}
\begin{proof}
We denote $\widetilde{v}=v\circ\Psi_T$ and $\widetilde{x}=\Psi_T^{-1}(x)\in T^m$ for $x\in T$. The chain rule gives
\[
D^2v(x)=D^2\Psi_T^{-1}(x)\nabla \widetilde{v}(\widetilde{x})+D\Psi_T^{-1}(x)D^2 \widetilde{v}(\widetilde{x}) D\Psi_T\Tinv(x)\,.
\]
For simplicity we omit the arguments $x$ and $\widetilde{x}$ in the following. A transformation shows 
\begin{align*}
\int_T D^2v:D^2w\dx-\int_{T^m} D^2\widetilde{v}:D^2\widetilde{w}\dxtilde=\uproman{1}+\uproman{2}+\uproman{3}+\uproman{4}
\end{align*}
with terms $\uproman{1},\uproman{2}$ and $\uproman{3}$ including $D^2\Psi_h^{-1}$
\begin{align*}
\uproman{1}&=\int_{T^m}D^2\Psi_T^{-1}\nabla \widetilde{v}:D^2\Psi_T^{-1}\nabla \widetilde{w}\det(D\Psi_T)\dxtilde\,,\\
\uproman{2}&=\int_{T^m}D^2\Psi_T^{-1}\nabla \widetilde{v}:D\Psi_T^{-1}D^2 \widetilde{w} D\Psi_T\Tinv\det(D\Psi_T)\dxtilde\,,\\
\uproman{3}&=\int_{T^m}D\Psi_T^{-1}D^2 \widetilde{v} D\Psi_T\Tinv:D^2\Psi_T^{-1}\nabla \widetilde{w}\det(D\Psi_T)\dxtilde\,,
\end{align*}
and the term $\uproman{4}$ excluding $D^2\Psi_T^{-1}$
\begin{align*}
\uproman{4}=\int_{T^m}D\Psi_T^{-1}D^2 \widetilde{v} D\Psi_T\Tinv:D\Psi_h^{-1}D^2 \widetilde{w} D\Psi_T\Tinv\det(D\Psi_T)\dxtilde
-\int_{T^m} D^2\widetilde{v}:D^2\widetilde{w}\dxtilde\,.
\end{align*}
Since $||D\Psi_T^{-1}||_{L^\infty(T)}$ and $||\det(D\Psi_T)||_{L^\infty(T^m)}$ are uniformly bounded with respect to $h$, estimate \eqref{interface_approx_inv} for $s=2$, the triangle and the Cauchy--Schwarz inequality yield
\begin{align*}
\left|\uproman{1}+\uproman{2}+\uproman{3}\right|\lesssim h_T^{m-1}||\widetilde{v}||_{H^2(T^m)}||\widetilde{w}||_{H^2(T^m)}\,.
\end{align*}
To treat the fourth term, we add and subtract $D\Psi^{-1}_TD^2\widetilde{v}D\Psi\Tinv_T:D\Psi^{-1}_TD^2\widetilde{w}D\Psi\Tinv_T$ as well as $D\Psi^{-1}_T D^2\widetilde{v}D\Psi\Tinv_T:D^2\widetilde{w}$ to arrive at
\begin{align*}
\uproman{4}&=\int_{T^m}D\Psi_T^{-1}D^2 \widetilde{v} D\Psi_T\Tinv:D\Psi_h^{-1}D^2 \widetilde{w} D\Psi_T\Tinv(\det(D\Psi_T)-1)\dxtilde\\
&\hspace{1cm}+\int_{T^m}D\Psi_T^{-1}D^2 \widetilde{v} D\Psi_T\Tinv:(D\Psi_h^{-1}D^2 \widetilde{w} D\Psi_T\Tinv-D^2\widetilde{w})\dxtilde\\
&\hspace{1cm}+\int_{T^m}(D\Psi_T^{-1}D^2 \widetilde{v} D\Psi_T\Tinv-D^2\widetilde{v}):D^2 \widetilde{w} \dxtilde\\[5pt]
&=\uproman{4}_a+\uproman{4}_b+\uproman{4}_c\,.
\end{align*}
By \eqref{interface_approx_det} and the same arguments as above we have
\begin{align*}
\left|\uproman{4}_a\right|\lesssim h_T^{m}||\widetilde{v}||_{H^2(T^m)}||\widetilde{w}||_{H^2(T^m)}\,.
\end{align*}
To deal with $\uproman{4}_b$ and $\uproman{4}_c$ we add an subtract respectively $D\Psi^{-1}_TD^2\widetilde{v}D\Psi_T\Tinv:D\Psi_T^{-1}D^2\widetilde{w}$ and $D\Psi^{-1}_TD^2\widetilde{v}:D^2\widetilde{w}$ to obtain
\begin{align*}
\uproman{4}_b&=\int_{T^m}D\Psi^{-1}_TD^2\widetilde{v}D\Psi_T\Tinv:D\Psi^{-1}_TD^2\widetilde{w}(D\Psi_T\Tinv-I)\dxtilde\\
&\hspace{2cm}+\int_{T^m}D\Psi^{-1}_TD^2\widetilde{v}D\Psi_T\Tinv:(D\Psi^{-1}_T-I)D^2\widetilde{w}\dxtilde\,,\\
\uproman{4}_c&=\int_{T^m}D\psi^{-1}_TD^2\widetilde{v}(D\Psi_T\Tinv-I):D^2\widetilde{w}\dxtilde+\int_{T^m}(D\psi^{-1}_T-I)D^2\widetilde{v}:D^2\widetilde{w}\dxtilde\,.
\end{align*}
From the identity 
\begin{align*}
||D\Psi_T^{-1}-I||_{L^\infty(T)}=||(D\Psi_T^{-1}-I)\T||_{L^\infty(T)}=||D\Psi_T\Tinv-I||_{L^\infty(T)}
\end{align*}
and estimate \eqref{interface_approx_inv} for $s=1$ we deduce
\begin{align*}
\left|\uproman{4}_b+\uproman{4}_c\right|\lesssim h_T^{m}||\widetilde{v}||_{H^2(T^m)}||\widetilde{w}||_{H^2(T^m)}\,.
\end{align*}
Combining all the results, using the norm equivalence \eqref{norm_equivalence} and $h^m_T\lesssim h^{m-1}_T$ yields \eqref{consistency_error}.
\end{proof}
\end{lem}

We next state the main result of this section.

\begin{thm}[Error estimate]\thlabel{interface_error}
Let $u\in \mathbb{A}(g,\Phi)$ be the solution to \eqref{interface_problem} and $\widetilde{u}_m\in \mathbb{A}_m(g,\Phi)$ be the solution to \eqref{interface_problem_approx}. With $u_m=\widetilde{u}_m\circ\Psi_T^{-1}$ we have
\begin{align}
||D^2(u-u_m)||_{L^2(\mathcal{T}_h)}\lesssim ||\widetilde{f}_m-\widetilde{f}\det(D\Psi_h)||_{L^2(\mathcal{T}_h^m)}+h^{m-1} ||\widetilde{u}_m||_{H^2(\mathcal{T}_h^m)}\,.
\end{align}
Upon choosing $\widetilde{f}_m=\widetilde{f}\det(D\Psi_h)$ we conclude
\begin{align*}
||D^2(u-u_m)||_{L^2(\mathcal{T}_h)}\lesssim h^{m-1}||\widetilde{u}_m||_{H^2(\mathcal{T}_h^m)}\,.
\end{align*}
\begin{proof}
Let $v=u_m-u$. By \eqref{consistency_error} we have
\begin{align*}
||D^2(u-u_m)||^2_{L^2(\mathcal{T}_h)}&=\int_{\mathcal{T}_h} D^2(u_m-u):D^2v\dx\\
&=\int_{\mathcal{T}_h} D^2u_m:D^2v\dx-\int_{\mathcal{T}_h} D^2u:D^2v\dx\\
&\lesssim \int_{\mathcal{T}_h^m} D^2\widetilde{u}_m:D^2\widetilde{v}\dxtilde-\int_{\mathcal{T}_h} D^2u:D^2v\dx\\
&\hspace{1cm}+h^{m-1}||\widetilde{u}_m||_{H^2(\mathcal{T}_h^m)}||\widetilde{v}||_{H^2(\mathcal{T}_h^m)}\\
&=\int_{\mathcal{T}_h^m}\widetilde{f}_m\widetilde{v}\dxtilde-\int_{\mathcal{T}_h}fv\dx+h^{m-1}||\widetilde{u}_m||_{H^2(\mathcal{T}_h^m)}||\widetilde{v}||_{H^2(\mathcal{T}_h^m)}\\
&\lesssim ||\widetilde{f}_m-\widetilde{f}\det(D\Psi_h)||_{L^2(\mathcal{T}_h^m)}||\widetilde{v}||_{L^2(\mathcal{T}_h^m)}+h^{m-1}||\widetilde{u}_m||_{H^2(\mathcal{T}_h^m)}||\widetilde{v}||_{H^2(\mathcal{T}_h^m)}\,,
\end{align*}
where we used that $u$ and $\widetilde{u}_m$ solve \eqref{interface_problem} and \eqref{interface_problem_approx}, respectively, followed by a transformation formula for the term including $f$ and Hölder's inequality. Due to the clamped boundary conditions, we have for every $\widetilde{v}\in \mathbb{A}_m(0,0)$ the Poincaré inequalities
\begin{align*}
||\widetilde{v}||_{L^2(\mathcal{T}_h^m)}\lesssim ||\nabla\widetilde{v}||_{L^2(\mathcal{T}_h^m)}\lesssim ||D^2\widetilde{v}||_{L^2(\mathcal{T}_h^m)}\,.
\end{align*}
This implies the equivalence of the $H^2$-norm $||\widetilde{v}||_{H^2(\mathcal{T}_h^m)}$ to the $H^2$-seminorm $||D^2\widetilde{v}||_{L^2(\mathcal{T}_h^m)}$. 
Combining the above results with the norm equivalence \eqref{norm_equivalence} yields
\begin{align*}
||D^2(u-u_m)||^2_{L^2(\mathcal{T}_h)}\lesssim \Bigl[||\widetilde{f}_m-\widetilde{f}\det(D\Psi_h)||_{L^2(\mathcal{T}_h^m)}+h^{m-1}
||\widetilde{u}_m||_{H^2(\mathcal{T}_h^m)}\Bigr]||D^2v||_{L^2(\mathcal{T}_h)}\,.
\end{align*}
Absorbing $||D^2v||_{L^2(\mathcal{T}_h)}$ on the left-hand side concludes the proof.
\end{proof}
\end{thm}

\begin{rmk}[Babuška's paradox]\thlabel{babuska}
The scaling $h^{m-1}$ of \thref{interface_error} is in agreement with the implication of Babuška's paradox \cite{Babuska}, in which a sequence of polygonal, nested domains $(\omega_n)_{n\in\mathbb{N}}$ converging to the open unit ball $\omega=B_1(0)\in\mathbb{R}^2$ is considered, i.e.,
\begin{alignat*}{2}
&\omega^n\subset\omega^{n+1}\subset\omega, \qquad&&\forall n\in\mathbb{N},\\
&\omega_n\to\omega, \qquad&&\text{as }n\to\infty\,,
\end{alignat*}
in the sense that for every point $x\in\omega$ there exists an index $n(x)>0$ such that $x\in\omega_n$ for all $n>n(x)$. The paradox states, that solutions $u_n$ of a Kirchhoff plate bending problem defined on $(\omega_n)_{n\in\mathbb{N}}$ with simple support boundary conditions fail to converge to the solution $u$ of the limit problem defined on $\omega$. Similarly, our error estimate indicates, that solutions defined for a piecewise linear approximation of the interface (m=1) may not converge to the solution of the exact interface problem.

\end{rmk}

%%%%%%%%%%%%%%%%%%%%%%%%%%%%%%%%%%%%%%%%%%%%%%%%%%%%%%%%%%%%%%%%%%%%%%%%%%%%%%%%%%%

\section{Numerical Experiments}\label{sec:experiments}

In this section we provide numerical results that were obtained using \textsc{Matlab}. These include experimental convergence rates that confirm the scaling of the a priori error estimate under suitably imposed boundary conditions. In addition, we compare the folding model to the classical bending model without a fold. The interface of the folding model is highlighted by a solid curve in the following figures. For all the simulations we use second order isoparametric elements ($k=2$) and $\gamma_0=\gamma_1=10$.

\subsection{Bending vs Folding} 
We first demonstrate the influence of the folding mechanism by comparing the classical bending model without a fold to the folding model with a quadratic interface, given by $\Gamma(y)=[2/3-2/3(y-y^2),y]^\top$ for different boundary conditions, $\Omega=(0,1)^2$ and $f=0$. The colors in the figures represent the elementwise bending energy density of the configurations.

For the first experiment we apply fully clamped boundary conditions on $\partial_D\Omega=\partial\Omega$, as shown in Figure \ref{fully_clamped_comp}. Large curvatures focus around the folding points on the boundary in both simulations. In addition, we observe singularities at the respective corners of the interface in the folding model. This demonstrates the fact, that the regularity of solutions depends on the geometry of the interface and its compatibility to the boundary conditions.

\begin{figure}[!h]
\begin{minipage}{.49\linewidth}
\centering
\includegraphics[width=0.9\linewidth]{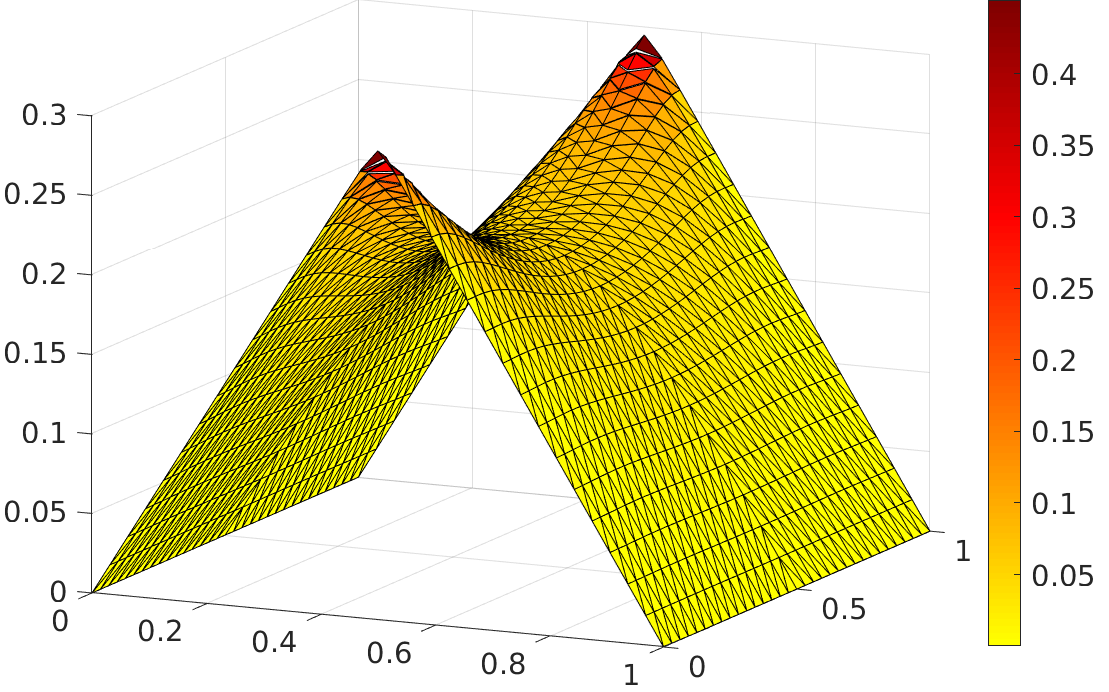}
\end{minipage}
\begin{minipage}{.49\linewidth}
\centering
\includegraphics[width=0.9\linewidth]{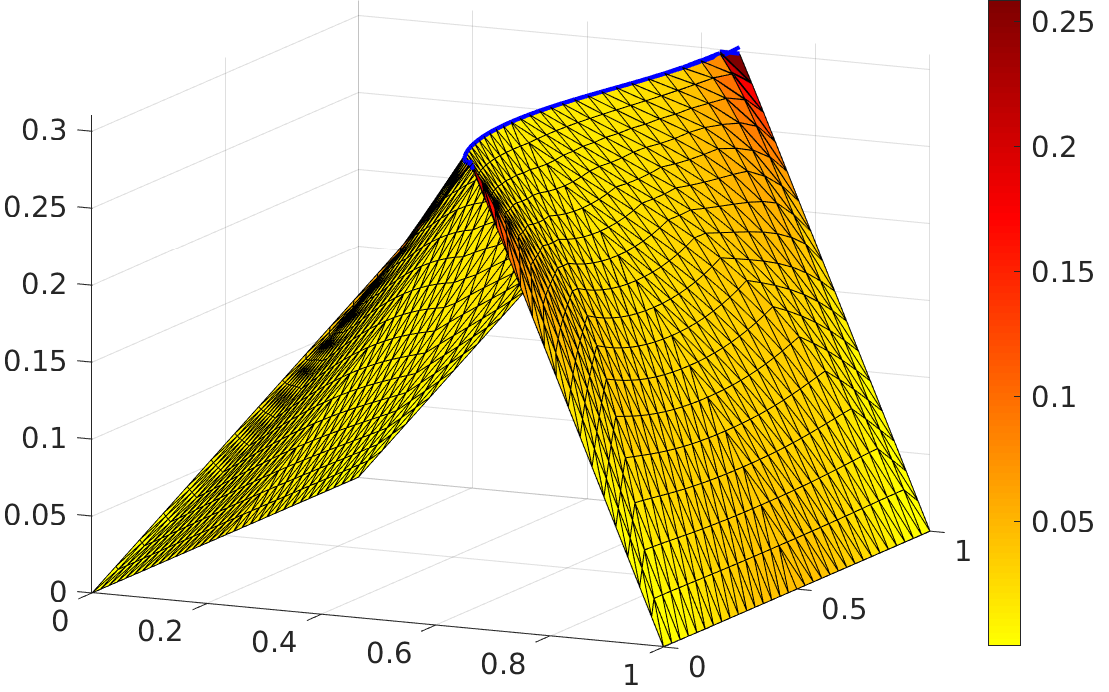}
\end{minipage}
\caption{Numerical simulation of the classical bending model without a fold (left) and the folding model with quadratic interface (right) for fully clamped boundary conditions.}
\label{fully_clamped_comp}
\end{figure}

Next we compare simulations for clamped boundary conditions $u=0$ and $\nabla u\neq0$ on the two sides $\{x=0\}$ and $\{x=1\}$, see Figure \ref{twosides_grad}. In the model without a fold, the curvatures are uniformly distributed along the plate with a slight increase in the direction of the plate center. In case of the folding model, high energy values are localized on the left side of the fold with maximum values occurring around the corners of the corresponding subdomain.

\begin{figure}[H]
\begin{minipage}{.49\linewidth}
\centering
\includegraphics[width=0.9\linewidth]{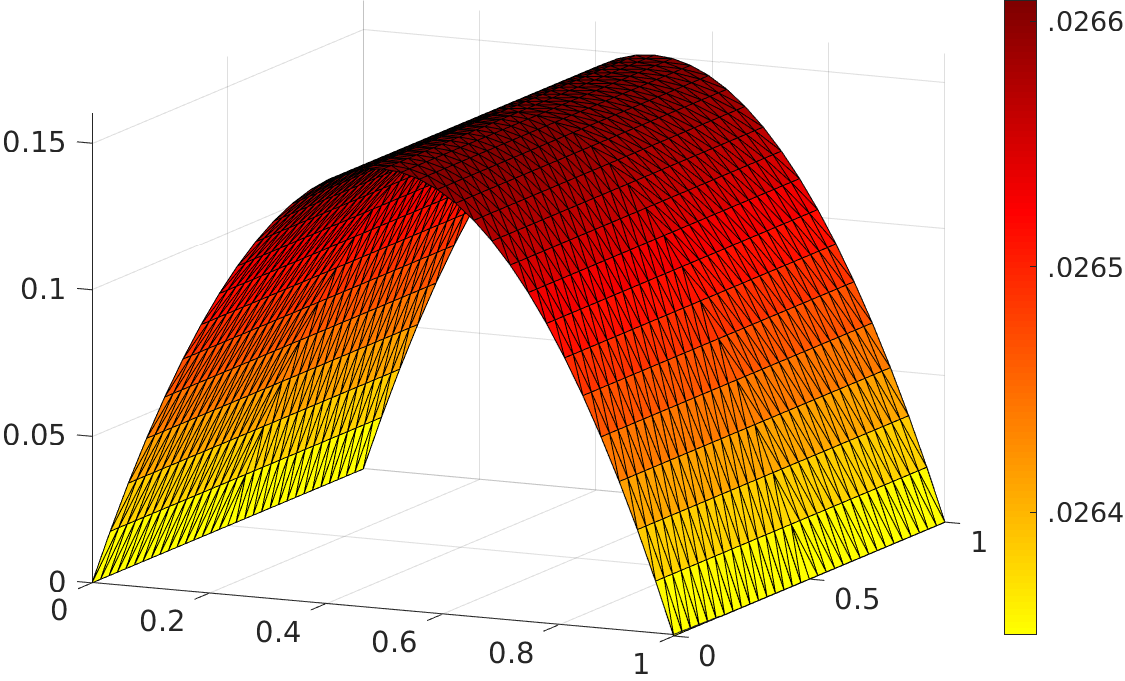}
\end{minipage}
\begin{minipage}{.49\linewidth}
\centering
\includegraphics[width=0.9\linewidth]{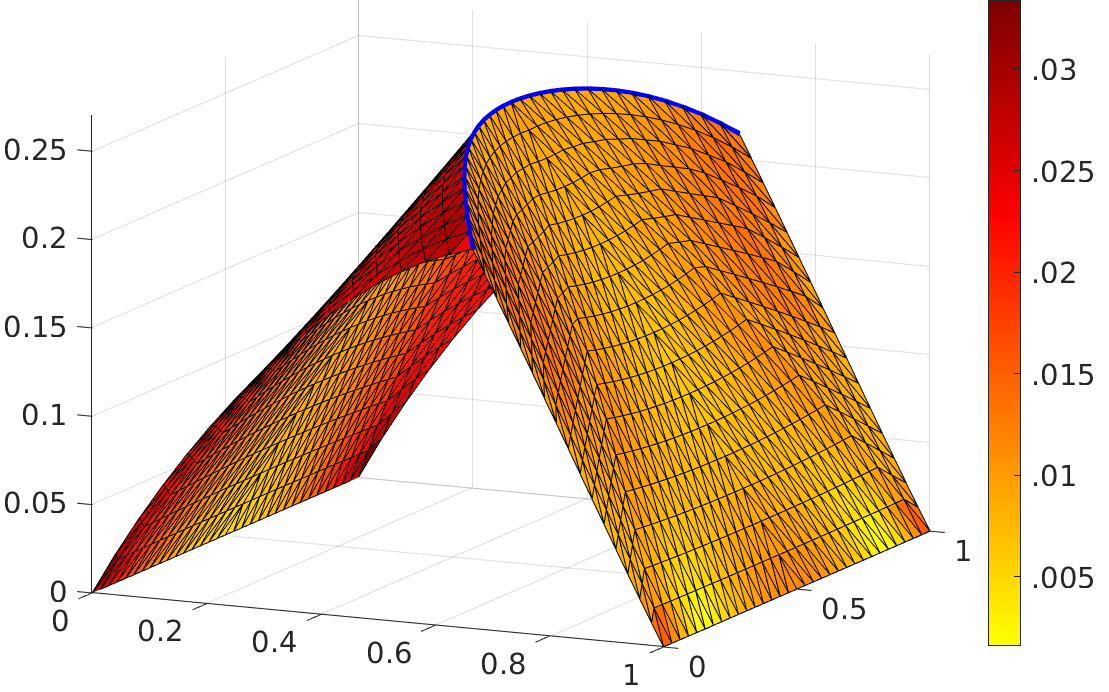}
\end{minipage}
\caption{Numerical simulation of the model without a fold (left) and the folding model with quadratic interface (right) for clamped boundary conditions on two sides of the plate.}
\label{twosides_grad}
\end{figure}

For the next simulation we apply clamped boundary conditions $u=0.3$ on $\{x=0\}$, $u=0$ on $\{x=1\}$ and $\nabla u=0$ on both sides, as can be seen in Figure \ref{twosides_zerograd}. The energy distributions of both configurations are very similar. In contrast to the previous simulation, large curvature values occur along the Dirichlet boundary with decreasing values in the direction of the plate center. Low curvature values are distributed along the fold, even around the endpoints of the interface.

\begin{figure}[H]
\begin{minipage}{.49\linewidth}
\centering
\includegraphics[width=0.9\linewidth]{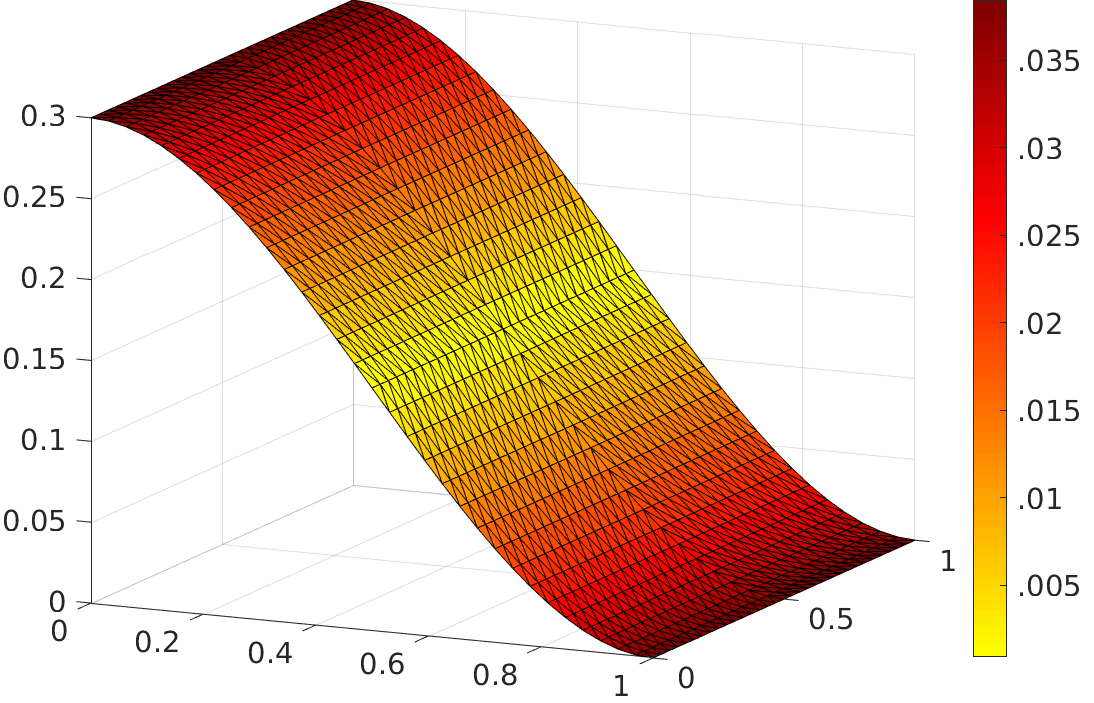}
\end{minipage}
\begin{minipage}{.49\linewidth}
\centering
\includegraphics[width=0.9\linewidth]{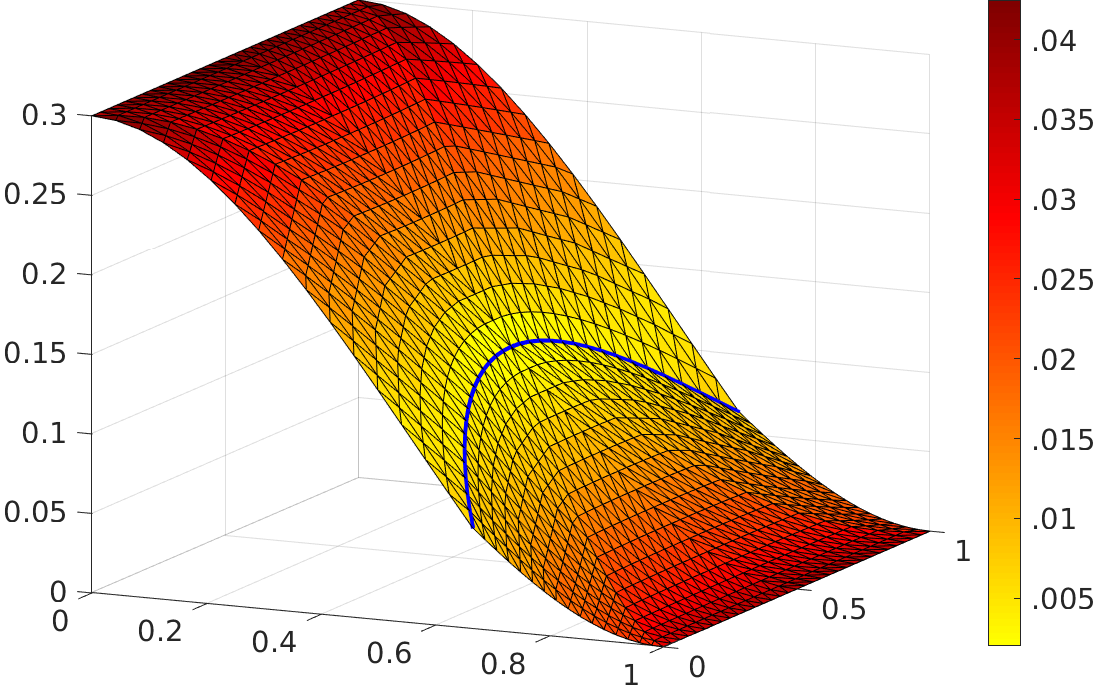}
\end{minipage}
\caption{Numerical simulation of the model without a fold (left) and the folding model with quadratic interface (right) for clamped boundary conditions on two sides of the plate.}
\label{twosides_zerograd}
\end{figure}

For the final simulation of this section, we apply clamped boundary conditions $u=0$ and $\nabla u=0$ on the right side $\{x\geq2/3\}\cap\partial\Omega$ (of the folding curve) and fix $u(x_0)=0.3$ for $x_0=[0,0.5]\T$ to ensure uniqueness of the folding configuration, as shown in Figure \ref{freeboundary}. Regarding the folding model, the deflection is very small on the right side of the interface with main bending effects occurring on the left side. Large curvatures concentrate around the endpoints of the fold. In striking contrast to the deformation without a fold, the folding deformation bends downwards towards the middle $\{y=0.5\}$ of the plate on the left side of the fold due to the shape of the interface.

\begin{figure}[H]
\begin{minipage}{.49\linewidth}
\centering
\includegraphics[width=0.9\linewidth]{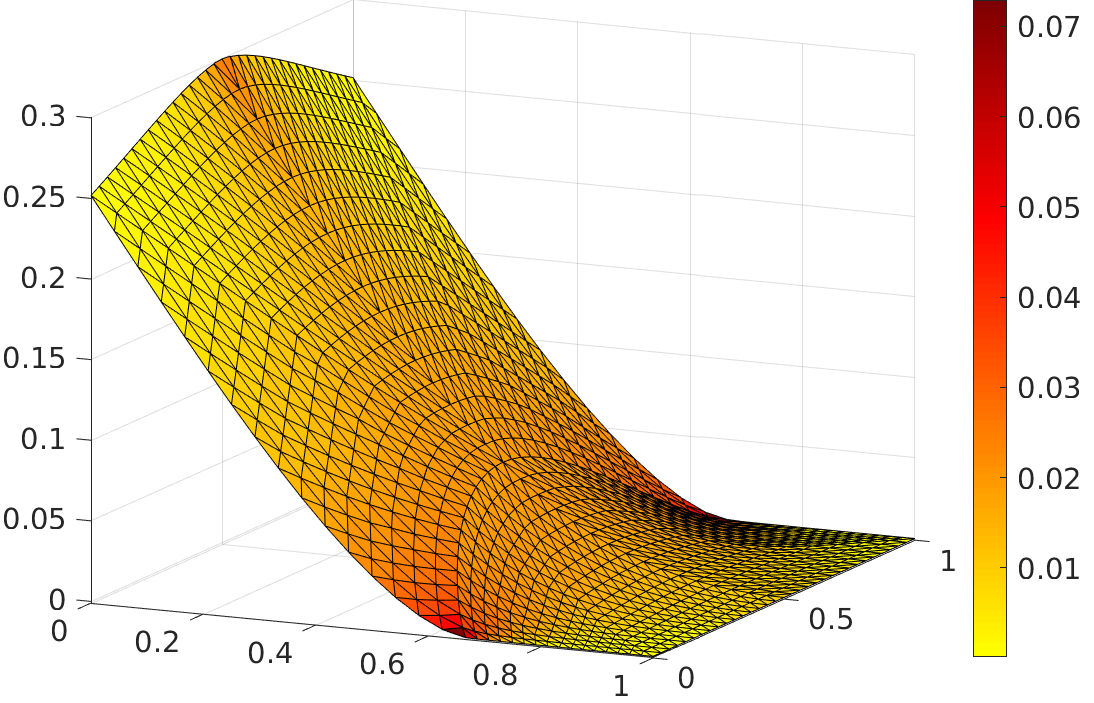}
\end{minipage}
\begin{minipage}{.49\linewidth}
\centering
\includegraphics[width=0.9\linewidth]{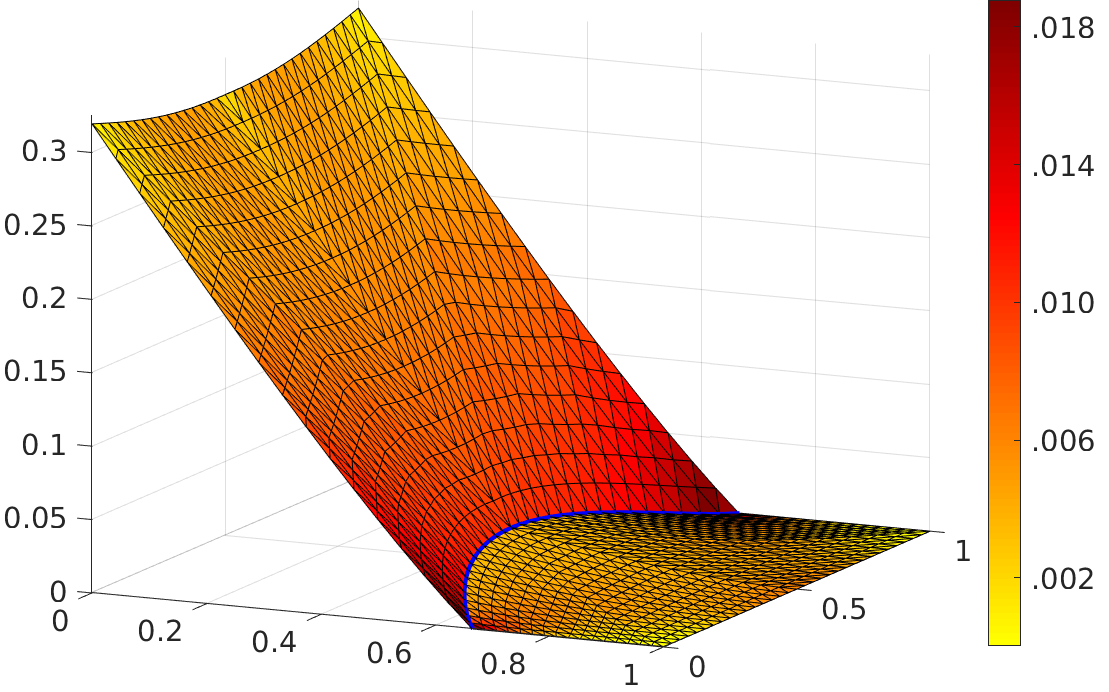}
\end{minipage}
\caption{Numerical simulation of the model without a fold (left) and the folding model with quadratic interface (right) for clamped boundary conditions on one side of the plate.}
\label{freeboundary}
\end{figure}

We observe from the simulations, that the shape of the folding curve has a great impact on the structure of the deformation and its energy distribution. Depending on the boundary conditions, large curvatures focus around the fold and the boundary of the plate. The shape of the interface dictates in which direction the plate bends on either side of the fold.

\subsection{Convergence Rates} Meaningful convergence rates require the availability of a piecewise regular solution to the continuous problem. The existence of such a solution depends on the geometry of the folding curve and its compatibility to the boundary conditions. To avoid possible single corner singularities (as can be observed in Figure \ref{fully_clamped_comp}), we consider clamped boundary conditions $g=\Phi=0$ on $\partial\Omega$ and a uniform force $f(x,y)=100$ on $\Omega$. We compare configurations and convergence rates for the model without a fold to the folding model with a straight interface at $\{x=0.5\}$, a piecewise linear and a piecewise quadratic interface, both approximating $\Gamma(y)=[2/3-1/6\sin(\pi y),y]^\top$ as $h\to0$, see Figure \ref{fully_clamped} and Table \ref{fully_clamped_conv_tab}.

\begin{figure}[H]
\centering
\begin{minipage}{.49\linewidth}
\includegraphics[width=0.9\linewidth]{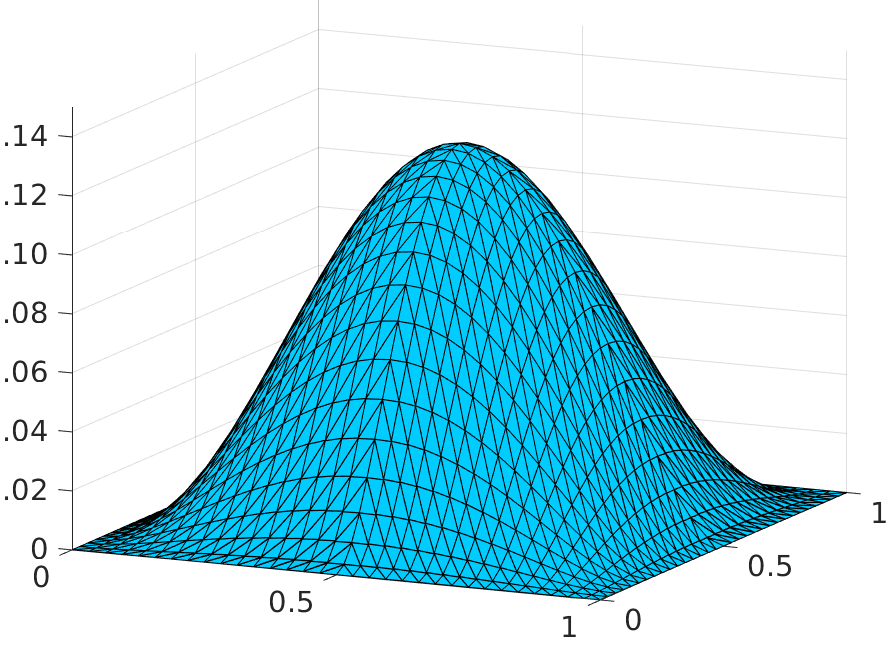}
\end{minipage}
\begin{minipage}{.49\linewidth}
\begin{tikzpicture}[scale=0.75]
\begin{axis}[
		xmode = log,
		ymode = log,
    xlabel={},
    ylabel={},
    ymax = 10^1.5,
    legend pos=north west,
    ymajorgrids=true,
    xmajorgrids=true,
    grid style=dashed,
    legend style={nodes={scale=0.9, transform shape}},
]
\addplot[
    color=red,
    mark=triangle,
    mark size = 3pt,
    ]
		table {error/NoFold/error_total_diff.txt};
    \addlegendentry{$||u-u_h||_{\text{dG}}$}
\addplot[
    color=green,
    mark=x,
    mark size = 3pt,
    ]
		table {error/NoFold/error_sides.txt};
    \addlegendentry{$||h^{-3/2}\dsquare{u_h}||_{\Gamma_h}$}
\addplot[
    color=blue,
    mark=diamond,
    mark size = 3pt,
    ]
		table {error/NoFold/error_grads.txt};
    \addlegendentry{$||h^{-1/2}\dsquare{\nabla_h u_h}||_{\Gamma_h\setminus\Gamma_h^{tr}}$}
\addplot[
		dashed,
    color=black,
    mark=none,
    line width = 0.9pt,
    domain = 0.011:0.707
    ]
    {5*x};
    \addlegendentry{$h$}  
\end{axis}
\end{tikzpicture}
\end{minipage}\vspace*{10pt}
\begin{minipage}{.49\linewidth}
\includegraphics[width=0.9\linewidth]{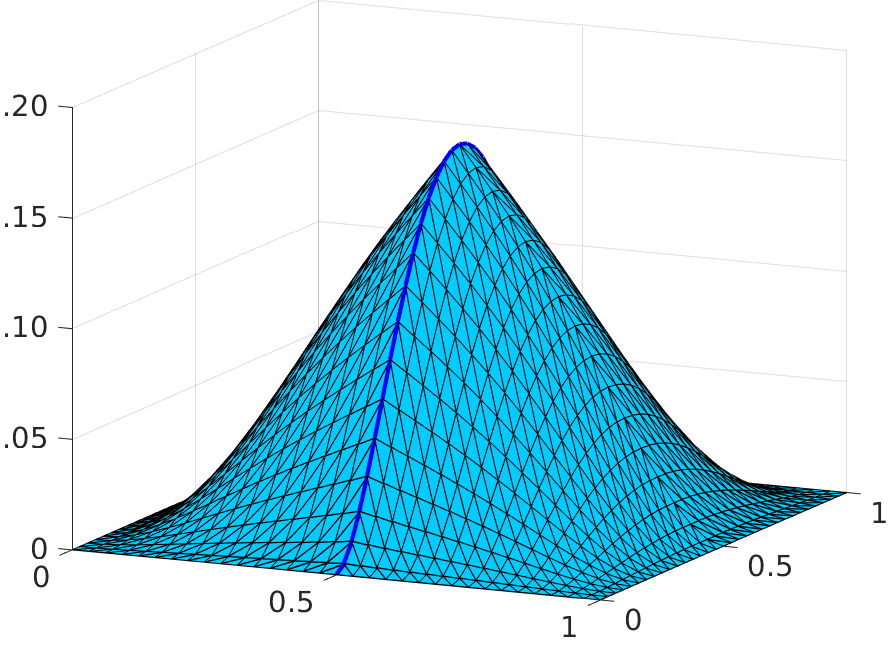}
\end{minipage}
\begin{minipage}{.49\linewidth}
\begin{tikzpicture}[scale=0.75]
\begin{axis}[
		xmode = log,
		ymode = log,
    xlabel={},
    ylabel={},
    ymax = 10^1.5,
    legend pos=north west,
    ymajorgrids=true,
    xmajorgrids=true,
    grid style=dashed,
    legend style={nodes={scale=0.9, transform shape}},
]
\addplot[
    color=red,
    mark=triangle,
    mark size = 3pt,
    ]
		table {error/StraightFold/error_total_diff.txt};
    \addlegendentry{$||u-u_h||_{\text{dG}}$}
\addplot[
    color=green,
    mark=x,
    mark size = 3pt,
    ]
		table {error/StraightFold/error_sides.txt};
    \addlegendentry{$||h^{-3/2}\dsquare{u_h}||_{\Gamma_h}$}
\addplot[
    color=blue,
    mark=diamond,
    mark size = 3pt,
    ]
		table {error/StraightFold/error_grads.txt};
    \addlegendentry{$||h^{-1/2}\dsquare{\nabla_h u_h}||_{\Gamma_h\setminus\Gamma_h^{tr}}$}
\addplot[
		dashed,
    color=black,
    mark=none,
    line width = 0.9pt,
    domain = 0.011:0.707
    ]
    {6*x};
    \addlegendentry{$h$}  
\end{axis}
\end{tikzpicture}
\end{minipage}\vspace*{10pt}
\begin{minipage}{.49\linewidth}
\includegraphics[width=0.9\linewidth]{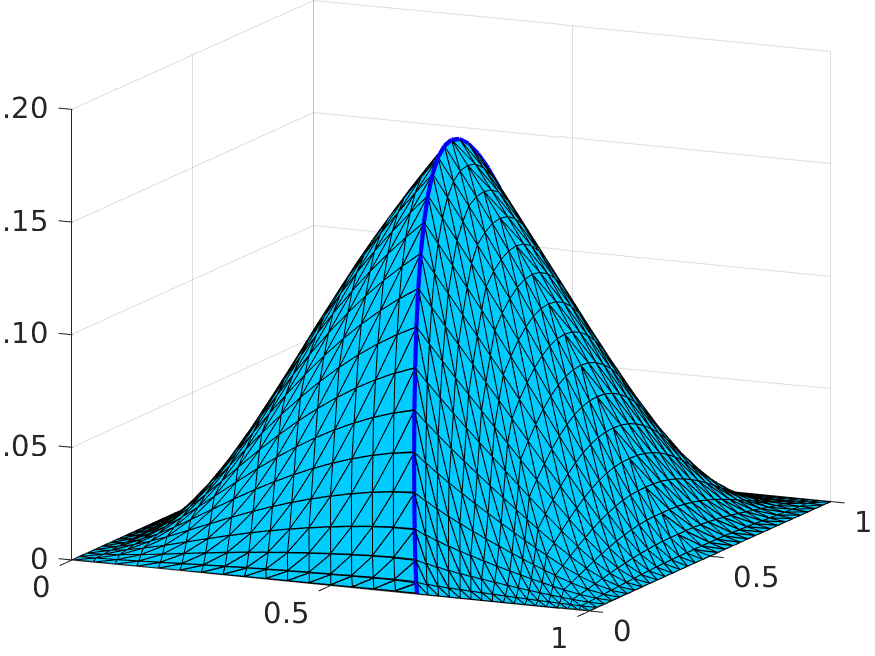}
\end{minipage}
\begin{minipage}{.49\linewidth}
\begin{tikzpicture}[scale=0.75]
\begin{axis}[
		xmode = log,
		ymode = log,
    xlabel={},
    ylabel={},
    ymax = 10^1.5,
    legend pos=north west,
    ymajorgrids=true,
    xmajorgrids=true,
    grid style=dashed,
    legend style={nodes={scale=0.9, transform shape}},
]
\addplot[
    color=red,
    mark=triangle,
    mark size = 3pt,
    ]
		table {error/PiecewiseLinearFoldSinus/error_total_diff.txt};
    \addlegendentry{$||u-u_h||_{\text{dG}}$}
\addplot[
    color=green,
    mark=x,
    mark size = 3pt,
    ]
		table {error/PiecewiseLinearFoldSinus/error_sides.txt};
    \addlegendentry{$||h^{-3/2}\dsquare{u_h}||_{\Gamma_h}$}
\addplot[
    color=blue,
    mark=diamond,
    mark size = 3pt,
    ]
    table {error/PiecewiseLinearFoldSinus/error_grads.txt};
    \addlegendentry{$||h^{-1/2}\dsquare{\nabla_h u_h}||_{\Gamma_h\setminus\Gamma_h^{tr}}$}
\addplot[
		dashed,
    color=black,
    mark=none,
    line width = 0.9pt,
    domain = 0.012:0.707
    ]
    {6*x};
    \addlegendentry{$h$}  
\end{axis}
\end{tikzpicture}
\end{minipage}\vspace*{10pt}
\begin{minipage}{.49\linewidth}
\includegraphics[width=0.9\linewidth]{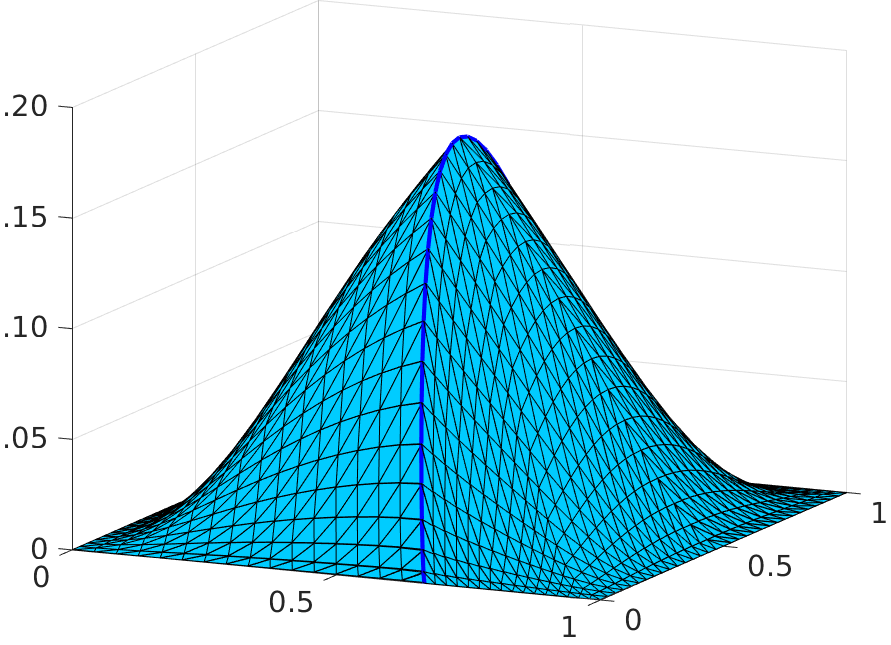}
\end{minipage}
\begin{minipage}{.49\linewidth}
\begin{tikzpicture}[scale=0.75]
\begin{axis}[
		xmode = log,
		ymode = log,
    xlabel={Meshsize h},
    ylabel={},
    ymax = 10^1.5,
    legend pos=north west,
    ymajorgrids=true,
    xmajorgrids=true,
    grid style=dashed,
    legend style={nodes={scale=0.9, transform shape}},
]
\addplot[
    color=red,
    mark=triangle,
    mark size = 3pt,
    ]
		table {error/CurvedFoldSinus/error_total_diff.txt};
    \addlegendentry{$||u-u_h||_{\text{dG}}$}
\addplot[
    color=green,
    mark=x,
    mark size = 3pt,
    ]
    table {error/CurvedFoldSinus/error_sides.txt};
    \addlegendentry{$||h^{-3/2}\dsquare{u_h}||_{\Gamma_h}$}
\addplot[
    color=blue,
    mark=diamond,
    mark size = 3pt,
    ]
    table {error/CurvedFoldSinus/error_grads.txt};
    \addlegendentry{$||h^{-1/2}\dsquare{\nabla_h u_h}||_{\Gamma_h\setminus\Gamma_h^{tr}}$}
\addplot[
		dashed,
    color=black,
    mark=none,
    line width = 0.9pt,
    domain = 0.012:0.707
    ]
    {6*x};
    \addlegendentry{$h$}  
\end{axis}
\end{tikzpicture}
\end{minipage}
\caption{Numerical solutions (left) and experimental error values (right) for the bending model without a fold and the folding model with a straight interface at $\{x=0.5\}$, a piecewise linear and a piecewise quadratic interface, both approximating $\Gamma(y)=[2/3-1/6\sin(\pi y),y]^\top$ (from top to bottom).}
\label{fully_clamped}
\end{figure}
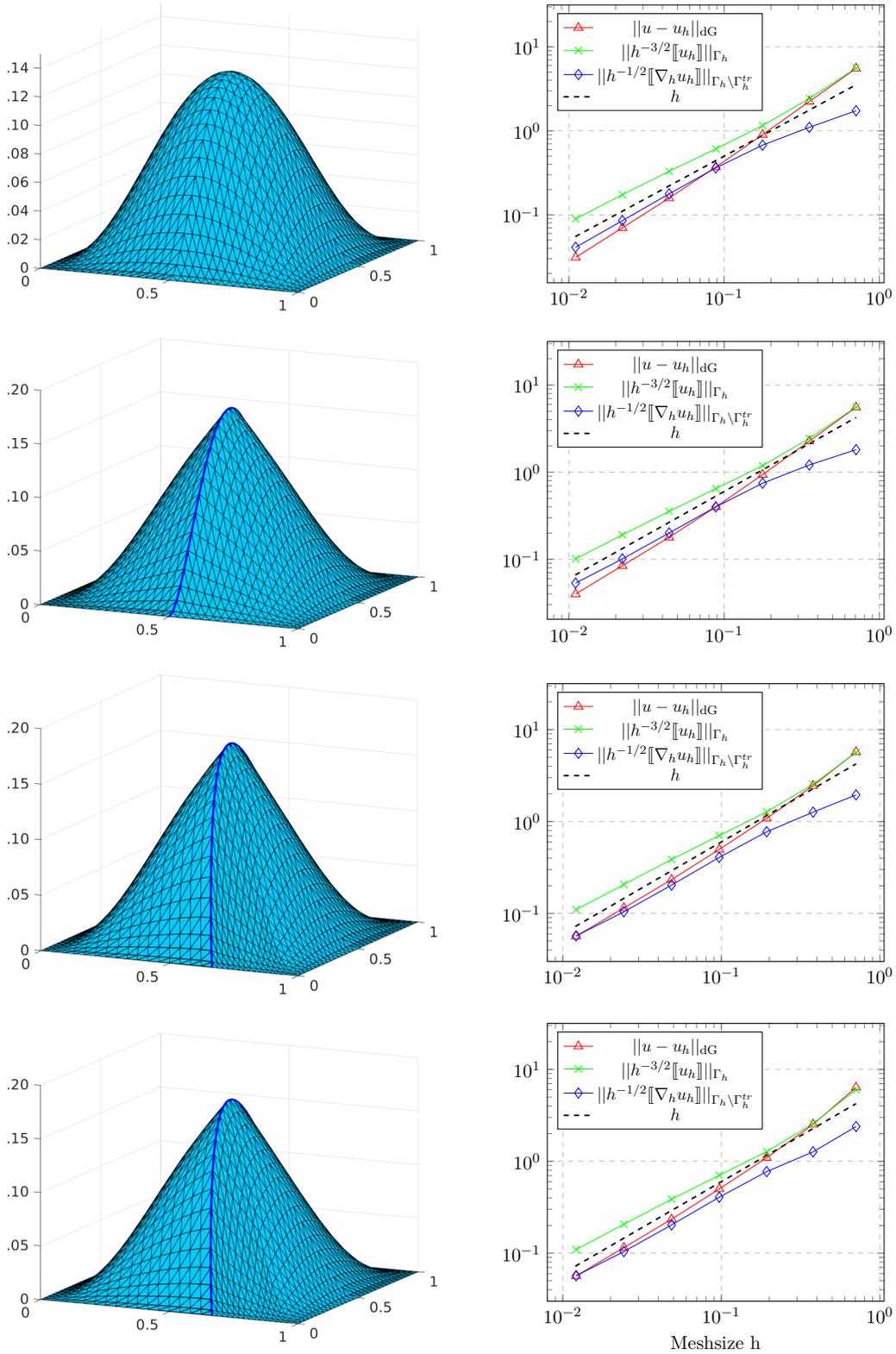

Since the solution $u$ to the continuous problem is unknown, we compute for a sequence of triangulations $(\mathcal{T}_j)_{j=0,1,...}$ and $s_j=||u_j||_{\text{dG}}$ the extrapolated values
\begin{align*}
\tilde{s}_j = \frac{s_js_{j-2}-s_{j-1}^2}{s_j-2s_{j-1}+s_{j-2}}\,,
\end{align*}
in order to approximate the value $s=||u||_{\text{dG}}$. The errors $||u-u_h||_{\text{dG}}$ are then approximated using the Galerkin orthogonality
\[
||u-u_h||_{\text{dG}}^2=||u_h||_{\text{dG}}^2-||u||_{\text{dG}}^2\,.
\]

\begin{table}[H]
\begin{center}
\begin{tabular}{ | r | c | c | c | c | } 
  \hline
  DoFs & No Fold & Linear Fold & Pw. Lin. Fold Appr.& Pw. Quadr. Fold Appr.\\ 
  \hline
  192 & 1.3020 & 1.2874 & 1.2190 & 1.3384 \\
  768 & 1.3158 & 1.2801 & 1.1767 & 1.2049 \\ 
  3072 & 1.2823 & 1.2292 & 1.1243 & 1.1281 \\ 
  12288 & 1.2262 & 1.1597 & 1.0916 & 1.0948 \\ 
  49152 & 1.1744 & 1.0828 & 1.0242 & 1.0261 \\ 
  196608 & 1.1737 & 1.0820 & 1.0230 & 1.0248 \\ 
  \hline
\end{tabular}
\end{center}
\caption{Experimental convergence rates of the $||\,.\,||_{\text{dG}}$-norm for the bending model without a fold and the folding model with a straight interface at $\{x=0.5\}$, a piecewise linear and a piecewise quadratic interface, both approximating $\Gamma(y)=[2/3-1/6\sin(\pi y),y]^\top$ (from left to right).}
\label{fully_clamped_conv_tab}
\end{table}

We observe a linear rate of convergence from Figure \ref{fully_clamped} and Table \ref{fully_clamped_conv_tab} in all four cases. In particular, the simulations confirm the theoretical scaling $h^{k-1}$ of \thref{error_estimate_thm} for $k=2$. A Babuška-like paradox could not be observed for the piecewise linear approximation of the interface. The geometric consistency error might be too small to be quantified in our simulation.

\section*{Acknowledgements}
The authors SB and PT acknowledge support by the DFG via the priority programme SPP 2256 \textit{Variational Methods for Predicting Complex Phenomena in Engineering Structures and Materials} (BA 2268/7-1). The author AB is partially supported by NSF grant DMS-2110811.

{\small\bibliography{Error_estimates_linear_folding_model}}

\Addresses

\end{document}